\documentclass[12pt]{amsart}

\newtheorem{theorem}{Theorem}[section]

\newtheorem{corollary}[theorem]{Corollary}

\newtheorem{lemma}[theorem]{Lemma}

\theoremstyle{definition}

\newtheorem{definition}[theorem]{Definition}

\newtheorem{remark}[theorem]{Remark}

\newtheorem{example}[theorem]{Example}

\theoremstyle{parrafo}

\numberwithin{equation}{theorem}

\begin{document}

\title[]{Optimal bounds on the modulus of continuity
 of the uncentered Hardy-Littlewood maximal function}

\author{J. M. Aldaz, L. Colzani and J. P\'erez L\'azaro}
\address{Departamento de Matem\'aticas,
Universidad Aut\'onoma de Madrid, 28049 Madrid, Spain.}
\email{jesus.munarriz@uam.es}
\address{Dipartimento di Matematica, Universit\`{a} di Milano-Bicocca,
 Edificio U5, via R. Cozzi 53, 20125 Milano, Italia.}
 \email{leonardo.colzani@unimib.it}
\address{Departamento de Matem\'aticas y Computaci\'on,
Universidad  de La Rioja, 26004 Logro\~no, La Rioja, Spain.}
\email{javier.perezl@unirioja.es}


\thanks{The first and third authors were partially supported by Grant MTM2009-12740-C03-03 of the
D.G.I. of Spain}

\thanks{This research was initiated during a visit of the first named author to the Universit\`a degli Studi di Milano-Bicocca, and continued during a visit by the third named author to the same university; both thank the Dipartimento di Matematica e Applicazioni for its
hospitality.}






\begin{abstract} We obtain sharp bounds for the modulus of continuity of
the uncentered maximal function in terms of the modulus of continuity of the given function, via integral formulas. Some of the results deduced
from these formulas are the following: The best constants for Lipschitz and H\"older
functions on proper subintervals of  $\mathbb{R}$ are $\operatorname{Lip}_\alpha ( Mf) \le (1 +
\alpha)^{-1}\operatorname{Lip}_\alpha( f)$, $\alpha\in
(0,1]$. On  $\mathbb{R}$, the
best bound for Lipschitz functions is $ \operatorname{Lip} ( Mf) \le
(\sqrt2 -1)\operatorname{Lip}( f).$ In higher dimensions, we determine the  asymptotic behavior, as $d\to\infty$, of the norm of the maximal operator associated to  cross-polytopes, euclidean balls and cubes, that is,
$\ell_p$ balls for $p = 1, 2, \infty$. We do this for arbitrary moduli of continuity. In the specific case of Lipschitz and
H\"older functions, the operator norm of the maximal operator is uniformly
bounded by $2^{-\alpha/q}$, where $q$ is the conjugate exponent of
$p=1,2$, and as $d\to\infty$ the norms approach this bound.
 When
$p=\infty$, best constants are the same as when $p = 1$.
\end{abstract}


\maketitle


\section
{Introduction.}

\markboth{J. M. Aldaz, L. Colzani, J. P\'erez L\'azaro} {Modulus of
continuity of the maximal function}

The constants appearing in the weak and strong type
inequalities satisfied by the Hardy-Littlewood maximal operator, in its
different variants,  have
been subject to considerable scrutiny. We mention, for instance,
\cite{CF}, \cite{Al1}, \cite{Me1}, \cite{Me2}, \cite{GMM}, \cite{GM}, 
 \cite{GK}, \cite{BD},  \cite{CLM}, \cite{St1}, \cite{St2}, \cite{St3},
\cite{Bou1}, \cite{Bou2},
\cite{Bou3}, \cite{Ca}, \cite{Mu}, \cite{StSt}, 
 \cite{Al2},  \cite{Al3},  \cite{AlPe4},  \cite{AlPe5}, \cite{NaTa},
  and the references contained therein.
Interest lies not only in determining sharp inequalities,
which in general
are hard to come by (in fact, no
best constants are known for dimensions larger than one) but also in finding out how  constants
change as certain parameters (for instance, the dimension) vary, or when the type 
of set one is averaging over
is modified, or the space of functions one is considering is changed.

Here we study the issue of optimal inequalities satisfied
by the uncentered maximal operator $M$, and also its asymptotic
behavior, but from a different viewpoint: Instead of considering weak and
strong type inequalities, we analyze the properties of $M$
in connection with the modulus of continuity of a function.
The overall emerging pattern reveals that the 
uncentered
maximal operator
improves regularity, by preserving moduli and reducing constants.
 But in general there is no ``qualitative" improvement in the
modulus. For instance, it may happen that there is no change in
the H\"older exponent of a function,
not even in a weak, almost everywhere   sense.
More precisely, we will see that there are H\"older functions $f$ with exponent
$\alpha \in (0,1)$, and no better than $\alpha$ on a set of positive
measure, such that $Mf$ is also no better than H\"older $(\alpha)$
 on a set of the same measure, cf. Example \ref{cantor}.
We note that the preservation of regularity  does not extend to  $C^1$ functions,
see Remark \ref{ex}.

This article is part of a wider
 project,
which attempts to find out under which conditions and to what extent
the Hardy-Littlewood maximal operator improves the regularity of
functions, in different settings. If there is such an improvement, then one can try to
prove variants of inequalities involving derivatives (for example,
Gagliardo-Niremberg-Sobolev type inequalities) with $DMf$ replacing
$Df$. In applications one often needs to consider functions more
general than those belonging to Sobolev spaces, so it is natural to
look for this kind of inequalities under as
little regularity as possible. It is also possible to consider other
maximal operators, associated to smoother approximations of the
identity, but we do not pursue this line of research here.

The study of the
Hardy-Littlewood maximal operator acting on spaces that measure
smoothness was initiated in \cite{Ki} by J. Kinnunen , who
 proved its boundedness on
 $W^{1,p}(\mathbb{R}^{d})$ for $1<p\le\infty$, and also
on the Lipschitz and H\"older classes (without increasing the
corresponding constants); see also \cite{KiLi}, \cite{HaOn},
\cite{KiSa}, \cite{Ta},
 \cite{Lu}. In \cite{AlPe}, part of the project outlined in the previous paragraph
is carried out for  functions of bounded
variation and $d=1$ (the situation when $d > 1$ is still not well
understood, cf. \cite{AlPe2} and \cite{AlPe3}).  Unlike the H\"older case, here a qualitative gain in regularity does occur (cf. Theorem 2.5): If $f$ is of bounded variation, then the
{\em uncentered} Hardy-Littlewood maximal function $Mf$ is absolutely continuous  (however,  the
centered maximal function need not even be continuous), and
furthermore,
the variation is not increased by $M$. As application, a Landau type inequality under less regularity is presented in Theorem 5.1 of \cite{AlPe}.

Given $(\mathbb{R}^d, \|\cdot\|)$, where $\|\cdot\|$ is an arbitrary
norm in $\mathbb{R}^d$, and $f:\mathbb{R}^d\to \mathbb{R}$,
 we present optimal integral formulas for  the
modulus of continuity of $Mf$ in terms of the modulus of $f$. Note that
distances appearing  in the moduli of continuity and balls defining
the maximal function, are determined according to $\|\cdot\|$.
However, we always use the same underlying $d$ dimensional Lebesgue measure, regardless of the norm under consideration, i. e., there are no different normalizations of the measure for different norms. We take the viewpoint that there is (essentially) only one
norm in dimension one, the usual absolute value, and this fixes Lebesgue measure in every
dimension by the requirement that a cube of sidelength one has measure
one. This also forces us to utilize the standard definition (via
the euclidean length) and the standard
normalizations, when dealing with Hausdorff measures. Such a convention (or some analogous consistency condition) is needed, for instance, to use
Fubini's Theorem, and more generally, the coarea formula.

A very brief exposition of the contents of this paper follows next.
The main integral
formulas of the paper, valid for an arbitrary norm in
the global case, where the domain under
consideration is the whole space $\mathbb{R}^d$, appear in Theorem \ref{mainabst} and Corollary \ref{opnormequa}.
These formulas are then specialized to the following three norms
and their associated maximal functions:
The $\ell_\infty$ norm (cf. Theorem \ref{cubes}),
the  $\ell_2$ norm (cf. Theorem \ref{2balls}), and the
 $\ell_1$ norm (cf. Theorem \ref{1balls}). Their associated maximal
 functions are respectively defined by averaging over cubes, euclidean balls,
 and cross-polytopes. A few consequences of these Theorems are the
 following: The  norm of the maximal operator acting on Lipschitz
 functions is $\left\|  M\right\|
_{\operatorname{Op}(1)}= \sqrt2 - 1$ for $d=1$, cf. Corollary
\ref{lip1}. For the $\ell_\infty$ norm the best constants
in dimensions 2 and 3 are approximately $0.574$ and $0.66155$. The exact values appear in Corollaries \ref{lip2}
and \ref{lip3}. For arbitrary $d$, $\|M\|_{\operatorname{Op}(1)} > (d - 1)/(d + 1)$, and the error in this estimate is
  $o (1/(d + 1))$, cf. Corollary
  \ref{alldinftybounds}. For the $\ell_1$ norm all constants
  are exactly the same as for the $\ell_\infty$ norm, see Theorem \ref{1ballsareeq}.
By way of contrast, we note that in the euclidean case
and on the class of Lipschitz functions,
$\|M\|_{\operatorname{Op}(1)} \le 2^{-1/2}$ in every dimension,
and this bound is optimal. With respect to the preceding results,
some open questions are mentioned; we indicate one now: Since 
$\left\| M f \right\|
_{\infty} = \left\| f\right\|
_{\infty}$ and (under the euclidean norm) $\left\| D M f \right\|
_{\infty} \le 2^{-1/2} \left\| D f\right\|
_{\infty}$, the maximal operator is a contraction on
$W^{1,\infty}(\mathbb{R}^d)$, the Sobolev space of essentially bounded
functions with essentially bounded derivatives (note that the contraction is not
strict, consider for instance the constant functions). Is there some analogous
result for $p < \infty$ sufficiently large? 

The last section of this paper deals with the maximal operator on
 proper subintervals of $\mathbb{R}$. In this case, the operator
norm of $M$ when acting on the H\"older and Lipschitz  classes is $\left\| M\right\|
_{\operatorname{Op}(\alpha)} = (1 + \alpha)^{-1}$
(cf.
Corollary \ref{lipholloc}). The local, higher
dimensional case is not studied here; we only point out that
constants when $d>1$ depend on the geometry of the domain, as was to
be expected, cf. Remark \ref{localhidee}.

Finally, we mention that the centered maximal operator $M^c$ associated to euclidean
balls satisfies $\left\|  M^c\right\|
_{\operatorname{Op}(\alpha)}=1$ in all dimensions, so Lipschitz and
H\"older constants are in general not reduced (cf. Remark
\ref{center}). Thus,  the preceding results, together with Theorem
2.5 of \cite{AlPe} mentioned above, suggest that from the viewpoint
of regularity the uncentered maximal operator is a more natural
object of study than the centered one.

\section
{Definitions and global  results for arbitrary norms.}

\begin{definition}\label{maxfun} Let $U\subset\mathbb{R}^d$ be an open set, let
$\|\cdot\|$ be a norm on $\mathbb{R}^d$, and let $B$ be a generic
ball with respect to this norm. Given  a locally integrable function
$f:U\to \Bbb R$, the noncentered Hardy-Littlewood maximal function
$Mf$  is defined by
$$
Mf(x) := \sup_{ x\in B\subset U}\frac{1}{|B|}\int_B |f(y)|dy.
$$
Here $|B|$ stands for the Lebesgue measure of $B$.
  Regarding the centered maximal function $M^c f(x)$, one
requires that  balls be centered at $x$ rather than just containing
it, but everything else is as in the uncentered case.
\end{definition}

\begin{definition}  The modulus of continuity of a function $f$ is
$$\omega\left(f,\delta\right)
:=\sup\left\{ \left|f(x)-f(y)\right|: \left\| x-y\right\|\leq\delta
\right\}.$$
\end{definition}

We point out that both definitions depend on the norm
under consideration (while the underlying Lebesgue
measure does not). Although  maximal functions and moduli of
continuity obtained via
different norms will always be pointwise comparable, from the viewpoint of best constants distinctions cannot be neglected. Indeed,
 when the
euclidean ($\ell_2$)  and the max ($\ell_\infty$) norms
are used in $\mathbb{R}^{d}, d > 1$ the results vary. Somewhat surprisingly,
since already in dimension 3 cubes and cross-polytopes are very different
geometrical objects, best constants associated
to the $\ell_\infty$ norm and arbitrary moduli of continuity
are exactly the same as those obtained under the $\ell_1$ norm.
Let us emphasize, though, that our results say nothing, for instance, about the
maximal function associated to cubes when the length used is the
euclidean distance. For us, fixing a norm fixes the balls
we average over; once we select, say, the euclidean norm, the maximal
function considered in our
theorems will be the one associated to euclidean balls.

The following theorem is due to Juha Kinnunen, cf. \cite{Ki}. While
not explicitly stated there, its proof appears in \cite{Ki}, pp.
120-121, Remark 2.2 (iii). A small variant of the argument yields
the boundedness in $W^{1,p}(\mathbb{R^d})$, $1 < p\le\infty$, of the
maximal operator (Remark 2.2 (i) of \cite{Ki});
 an abstract version can
be found in Theorem 1 of \cite{HaOn}, where it is applied to the
spherical maximal operator.

In the next theorem $\mathcal{M}$
denotes a generic Hardy-Littlewood maximal operator (it could be
centered or noncentered, associated to euclidean or to other balls).

\begin{theorem}\label{kin} {\bf (J. Kinnunen)}. Let  $f:\mathbb{R}^d\to \Bbb R$ be locally integrable. Then, for every
$\delta>0$,
\begin{equation}\label{trivbd}
\omega\left(\mathcal{M}f,\delta\right)\leq \omega\left(|f|,\delta
\right).
\end{equation}
\end{theorem}
\begin{proof}
Let $h, x\in\mathbb{R}^d$ and set $f_h(x):=f(x+h)$. By commutativity of $\mathcal{M}$ with
translations, $\mathcal{M}f(x+h)=\mathcal{M}f_h (x)$, and by subadditivity, $|\mathcal{M}f_h-\mathcal{M}f|\le
\mathcal{M}(|f_h|-|f|)$. Since averages never exceed a supremum,
we have $\sup_{x\in\mathbb{R}^d}
\mathcal{M}(|f_h|-|f|) (x) \le \sup_{x\in\mathbb{R}^d} (\left||f_h|-|f|\right|) (x)$,
and now (\ref{trivbd})
follows by
taking the sup  over $\|h\|\le \delta$.
\end{proof}

\begin{remark}\label{center} This simple result
already contains the sharp bound for the centered maximal operator
$M^c$ on the Lipschitz and H\"older classes. Let $\psi(x)=
\max\{1-|x|, 0\}$ be defined on $\mathbb{R}$;  clearly, $M^c \psi
(x) = \psi(x)$ for every $x\in [-1/2, 1/2]$, so both functions have the
same Lipschitz constant:
 $\operatorname{Lip} (M^c\psi) = \operatorname{Lip} (\psi)= 1$.   Now, this
example can be easily adapted to higher dimensions (for instance, by
making the corresponding function depend only on the first
coordinate), and it follows that the centered maximal operator on
$\mathbb{R}^d$ does not in general reduce the Lipschitz constant of
a function. By way of contrast, we mention that the uncentered
maximal operator on $\mathbb{R}$ satisfies $\operatorname{Lip}
(Mf) \le (\sqrt2 -1)\operatorname{Lip} (f)$, cf. Corollary
\ref{lip1} below, while on $\mathbb{R}^d$ with the Euclidean norm,
the bound $\operatorname{Lip} (Mf) \le 2^{ -1/2}\operatorname{Lip}
(f)$ holds uniformly in the dimension (Theorem \ref{2balls}).

As for H\"older functions, the preceding example can be easily
modified to yield the same conclusion. In one dimension, set $\eta
(x) := \max\{1-|x|^\alpha, 0\}$ when $x \le 0$,  set $\eta (x) :=
1+|x|^\alpha$ on $[0,1/4]$, and finally, extend $\eta$ to $[1/4,
\infty)$ by reflection about the $x=1/4$ axis. Then, by
concavity of $\eta$ on $0<x<1/4$,  $M^c \eta (x)
= \eta(x)$ for every  $x \in [0, 1/4]$, and thus both functions have
the same H\"older constant (we mention that the uncentered operator improves 
H\"older constants, see
formula (\ref{1d}) below). It is clear that this example can
also be
adapted to higher dimensions.
\end{remark}

 Our aim in this paper is to find best inequalities in the spirit of (\ref{trivbd})
 for the uncentered maximal operator, which, as noted in the introduction, has better properties regarding the regularization
 of functions than its centered relative.
\begin{definition} A function $\omega: [0,\infty )\to[0,\infty )$ is a modulus of continuity if
it is the modulus of continuity of a uniformly continuous function,
i.e., if there is a uniformly continuous function $f$ such that for
all $\delta \ge 0$, $\omega\left(\delta\right) =
\omega\left(f,\delta\right)$.
\end{definition}

\begin{remark} Often a modulus of continuity $\omega: [0,\infty )\to[0,\infty )$ is defined
as a continuous, nondecreasing, subadditive function (so $\omega (a +
b) \le \omega (a) +\omega (b)$), vanishing at zero. Note that from
continuity and subadditivity it follows that  $\omega$ is uniformly
continuous. It is well known and not difficult to check that the
modulus of continuity of a uniformly continuous function $f$ has
these properties, while given any $\omega$ satisfying the above
conditions, there is a uniformly continuous $f$ such that
$\omega(\cdot) = \omega (f,\cdot)$, namely $\omega$ itself. We do
not assume that moduli of continuity are bounded, so our results
apply to general Lipschitz and H\"older functions. To avoid trivialities,
given a generic modulus of continuity, we shall assume it is not identically
zero.
\end{remark}

The next theorem, and its corollary \ref{opnormequa}, contain the main integral formulas of the paper,
valid (in the global case) for all norms and all dimensions.  Note that
on  $\mathbb{R}^{d}$ the value of $Mf(x)$ is the same regardless of whether the
balls we average over are taken to be open or closed. We shall assume
whatever is more convenient at any given point. For instance, in  the next theorem and its proof we suppose that balls are closed. Additionally, we use the following notation: $a_n\uparrow b$ (resp. $a_n\downarrow b$) means that the sequence $\{a_n\}$ converges to $b$ in a monotone increasing (resp. decreasing) fashion.

\begin{theorem} \label{mainabst} Let $d \ge 1$ and let $f$ be a locally
integrable function on $\mathbb{R}^{d}$. Then, for every norm
$\|\cdot\|$ on $\mathbb{R}^{d}$ and every $t\geq0$,
\begin{equation}\label{abstform}
\omega\left(Mf,t\right)  \leq\\
\sup_{\left\{v\in\mathbb{R}^{d}:\;\left\|  v\right\|  = 1\right\}}
\inf_{\left\{c\in\mathbb{R}^{d},R>0 :\left\|  v-c\right\| \le
R\right\}}\dfrac{1}{\left|  B(0,1)\right|}
{\displaystyle\int_{B(0,1)}} \omega\left(  |f|,t\left\|
c+Ru\right\|  \right)  du.
\end{equation}
The preceding inequality is optimal in the  sense that given a modulus of continuity $\omega$, there exists a function $\psi$
such that letting $f = \psi$ in (\ref{abstform}), the following equalities hold: For all $t > 0$, we have $\omega\left(\psi,t\right)  = \omega\left(t\right)$,
and furthermore,
\begin{equation}\label{equality}
\omega\left(M\psi,t\right)  =
\sup_{\left\{v\in\mathbb{R}^{d}:\;\left\|  v\right\|  = 1\right\}}
\inf_{\left\{c\in\mathbb{R}^{d},R>0 :\left\|  v-c\right\| \le
R\right\}}\dfrac{1}{\left|  B(0,1)\right|}
{\displaystyle\int_{B(0,1)}} \omega\left( t\left\|
c+Ru\right\|  \right)  du.
\end{equation}
If  $v$ is a unit vector, to find its associated infimum in  (\ref{equality}) it is enough
to consider the set
\begin{equation}\label{infset2}
\left\{c\in\mathbb{R}^{d},R>0 :\left\|  v-c\right\| = R
\text{ and }R \le 1 \right\}.
\end{equation}
As for the extremal functions $\psi$ appearing in (\ref{equality}),
they have the following form: If $\omega$ is bounded, we can take
$\psi(x)=\|\omega\|_{L^\infty([0,\infty))}-\omega(\|x\|)$, while if
$\omega$ is unbounded, we set
$\psi(x)=\sum_{n=1}^\infty(\Omega_n-w(\|x-a_n\|))^+$, where
$\{\Omega_n\}_1^\infty\uparrow\infty$  and  $\{a_n\}_1^\infty$  diverges to infinity so
fast, that on the support of each spike the maximal function does not depend on any
of the other spikes.
\end{theorem}

Observe that (\ref{abstform}) is stronger, for the uncentered maximal
operator, than Kinnunen's inequality $\omega\left(  Mf,t\right)
\leq\omega\left( |f|,t\right)$, as can be seen by letting
$R\rightarrow0$ (and hence $c\rightarrow v$). Note also that the
right hand side of (\ref{equality}) depends only on the modulus of
continuity and
on the norm. Regarding the extremal functions,  our strategy
to find them is easy to explain (and the proof of the theorem shows that
it works): Suppose $\omega$ is a bounded modulus of continuity,
and suppose that a given function $g$ has a global maximum at $x$.
Then $Mg (x) = g (x)$, and in order to maximize
$Mg (x) - M g (y)$ for each $y\ne x$, we want to minimize $Mg (y)$.
Thus $g$
should have the fastest
possible decay allowed
by $\omega$ in every direction (and hence the least possible mass).
 This is precisely what  $\psi(x):= \|\omega\|_{L^\infty([0,\infty))}-\omega\left(  \left\|  x\right\| \right)$ does.
 A similar observation (in a local sense) can be made for unbounded moduli.  Finally, we mention that while it is natural to suspect that  the centers $c$ in
(\ref{infset2}) associated to $v$ should be chosen so that $B(c,R) \subset B(0,1)$,
 Theorem \ref{mainabst} does not make any such assertion, and in fact, it
 may be
false in general. But
we shall see in the next section that it is indeed true for the
$\ell_1, \ell_2$ and $\ell_\infty$ norms.

\begin{proof}  Given $f \ge 0$, we assume that the associated modulus of continuity
$\omega\left(  f,t\right)$ is finite for every $t>0$, for otherwise
there is nothing to prove. Hence, we take $f$ to be locally bounded.
We may also assume that $Mf$ is not constant, and in particular,
that it is not identically $\infty$. Next, choose
$x,y\in\mathbb{R^d}$ and suppose that $Mf(y)<Mf(x)$. Set
$$
E:=\{b\in\mathbb{R}^d,T>0: \|(y-x)-b\|\le T \} \mbox{ \ \ and \ \ }
F:=\{a\in\mathbb{R}^d,S>0: \|x-a\|\le S \}.
$$
 Then, for all $(a,S)
\in F$ we have $  E\subset\{b\in\mathbb{R}^d,T>0: y\in B(a+b,S+T)\}$
by the triangle inequality. Hence
\begin{equation*}
  Mf(x)-Mf(y)= \sup_{F}\frac{1}{|B(a,S)|}\int_{B(a,S)}f(u)du -Mf(y)
\end{equation*}
\begin{equation*}
  \le \sup_F\inf_E\left(\frac{1}{|B(a,S)|}\int_{B(a,S)}f(u)du
-\frac{1}{|B(a+b,S+T)|}\int_{B(a+b,S+T)}f(u)du\right)
\end{equation*}
\begin{equation*}
\le \sup_F\inf_E
\left(\frac{1}{|B(0,1)|}\int_{B(0,1)}|f(a+Su)-f(a+b+(S+T)u)|du\right)
\end{equation*}
\begin{equation*}
\le \sup_F\inf_E
\frac{1}{|B(0,1)|}\int_{B(0,1)}\omega(f,\|b+Tu\|)du\end{equation*}
\begin{equation*}
=\inf_E \frac{1}{|B(0,1)|}\int_{B(0,1)}\omega(f,\|b+Tu\|)du,
\end{equation*}
where the $\sup_F$  has been deleted from the last line since neither  $a$ nor $S$
appear in the integral. By symmetry of $B(0,1)$, the inequality
\begin{equation}\label{dif}
  | Mf(x)-Mf(y)|\le \inf_E
\frac{1}{|B(0,1)|}\int_{B(0,1)}\omega(f,\|b+Tu\|)du
\end{equation}
also holds  when $Mf(y)>Mf(x)$, and thus it always holds. Writing in
(\ref{dif})
$c\|y-x\|=b$, $R\|y-x\|=T$, and $y-x=\|y-x\|v$, where $\|v\|=1$,
  we have
\begin{equation*}
  |Mf(x)-Mf(y)|\le
\inf_{\{c\in\mathbb{R}^d,R>0:
\|v-c\|\le R \}}
  \frac{1}{|B(0,1)|}\int_{B(0,1)}\omega(f,\|y-x\|\|c+Ru\|)du.
\end{equation*}
Since the right hand side of the last inequality is increasing  in
$\|y-x\|$, it follows that for every $t>0$,
\begin{equation*}
  \omega(Mf,t)=\sup_{\{x,y\in\mathbb{R}^d:\|x-y\|\le
  t\}}|Mf(x)-Mf(y)|
\end{equation*}
\begin{equation}\label{supinf}
\le \sup_{\{v\in\mathbb{R}^d:\|v\|=1\}}\inf_{\{c\in\mathbb{R}^d,R>0:
\|v-c\|\le R \}}
  \frac{1}{|B(0,1)|}\int_{B(0,1)}\omega(f,t\|c+Ru\|)du.
\end{equation}

Next we prove that this inequality is sharp.
Suppose first that $\omega$ is a bounded modulus of continuity. Set $\Omega := \|\omega\|_{L^\infty([0,\infty))}$, and write
 $\psi(x):= \Omega-\omega\left(  \left\|  x\right\| \right)$. Fix $t > 0$. If $x$ satisfies $0
<\left\|x\right\| \le t$, then we can express $x=tv$, where $0<\|v\|\le
1$, and we get
\begin{gather*}
\omega(M \psi, t) \ge   M\psi(0)-M\psi(x) =\inf_{\left\{
c\in\mathbb{R}^{d},R>0 : \left\|  v-c\right\| \le R\right\}}
\dfrac{1}{\left|  B(0,1)\right|} {\displaystyle\int_{B(0,1)}}
\omega\left(  t\left\|  c+Ru\right\|  \right)  du.
\end{gather*}
Now the result follows  by taking the supremum over all $v$
such that $0<\|v\|\le
1$.
If $\omega$ is unbounded,
 we modify $\psi$ it as follows: Take
 a sequence of suitably chosen ``spike" functions
$\left(\Omega_n-\omega\left(  \left\|  x- a_n\right\| \right) \right)^+$, and then set
$\psi(x)=\sum_{n=1}^\infty(\Omega_n-w(\|x-a_n\|))^+$.
``Suitably chosen" in the preceding sentence means that the
 heights $\Omega_n$ tend to infinity and the different
spikes are placed so far apart (i.e.,  $a_n\to\infty$ so fast) that on the support of each spike, the
others need not be taken into account when computing $M \psi (x)$
(in particular, different spikes will have disjoint
supports, so there are no convergence issues with the series
defining $\psi$).

In order to  apply formula (\ref{supinf}), it is useful to narrow
down as much as possible where  the infimum occurs. First we show
that it is enough to consider $R\le 1$. Fix a unit vector $v$ and suppose $R >1$. Since $\|v - 0\| = 1 \le R$, the origin is an admissible center
$c$ associated to $v$.
But then averaging $\omega(f, t\|\cdot\|)$ over
$B(0,1)$ yields a  value no larger than averaging over any
other ball $B(c, R)$ containing $v$, for  every
vector in $B(c,R)\setminus B(0,1 )$ has  norm larger than one, hence
larger than the norm of
any vector in $B(0, 1)\setminus B(c,R )$, and additionally
$|B(c,R)\setminus B(0,1 )| > |B(0,1)\setminus B(c,R )|$.

 Next we prove that it is enough to consider
pairs $(c, R)$ for which $\|v-c\|= R$.
Suppose $\|v-c\|< R$. Since
the continuous function
$r\mapsto rR-\|v-rc\|$ changes sign on $[0,1]$,
there exists an $r_0\in [0, 1]$ such that
$\|v-r_0 c\|=r_0 R$ and, for such an $r_0$, we have $\omega(f, t\|r_0 c+r_0 Ru\|)\le
\omega(f,t \|c+Ru\|)$. Hence, we do no worse by using the pair
$(r_0 c, r_0 R)$ instead of $(c, R)$, and
 (\ref{infset2}) follows.
\end{proof}

\begin{remark} \label{anderson}
 We mention that to obtain
  (\ref{infset2}), instead of the
  selfcontained argument given above, we could have applied
Anderson's Theorem (cf., for instance, Theorem 1.11,
pg. 376 of \cite{Ga}, or Theorem 1 of \cite{An}), which for  nonnegative, integrable, symmetric and
unimodal functions $f$, and origin symmetric convex bodies
$K\subset\mathbb{R}^d$,
tells us that
$$ {\displaystyle\int_{K}}
f\left( x + cy  \right)  dx \ge {\displaystyle\int_{K}}
f\left( x + y  \right)  dx,$$
where $0 \le c \le 1$ and $y\in\mathbb{R}^d$. Since we shall use
Anderson's Theorem later on, it is stated here for easy reference.
\end{remark}

\begin{definition} \label{lipopnorm} Given a modulus of continuity $\omega$,
 define the Lipschitz
space $\operatorname{Lip}(\omega) = \operatorname{Lip}(\omega, X)$ via the seminorm
\[
\left\|  f\right\|  _{\operatorname{Lip}(\omega, X)}:=\sup_{\{x, y \in X :x\neq y\}}
\frac{\left| f(x)-f(y)\right|  }{\omega\left(  \left\| x-y\right\|
\right) } =\sup_{ t > 0}
\frac{\omega ( f, t)}{\omega\left(t\right) }.
\]
Then $f\in \operatorname{Lip}(\omega, X)$ if (and only if)
$\left\|  f\right\|  _{\operatorname{Lip}(\omega, X)} < \infty.$
In this section and the next we will always have $X = \mathbb{R}^d$,
so reference to $X$ shall usually be omitted. Next,  set
\begin{equation}\label{opnorm}
\left\|  M\right\|  _{\operatorname{Op}(\omega)}
:=\sup_{\left\|
f\right\| _{\operatorname{Lip}(\omega)}\ne 0}\frac{\left\|  Mf\right\|
_{\operatorname{Lip}(\omega )}}{\left\|
f\right\| _{\operatorname{Lip}(\omega)}}
 =\sup_{\left\|
f\right\| _{\operatorname{Lip}(\omega)}=1}\left\|  Mf\right\|
_{\operatorname{Lip}(\omega )}.
\end{equation}

When $\omega\left(  t\right)  =t^{\alpha}$ and $0<\alpha \le 1$, we
use $\operatorname{Lip_\alpha}(X)$ (or just $\operatorname{Lip}
(\alpha)$) to denote the corresponding spaces of H\"{o}lder
continuous functions and of
 Lipschitz functions on $X$,  $\operatorname{Lip}_\alpha (f)$ to denote $\left\|  f\right\|_{\operatorname{Lip}(\omega)}$, and
 $\left\|  M\right\|_{\operatorname{Op}(\alpha)}$ to denote
 $\left\|  M\right\|  _{\operatorname{Op}(\omega)}$. If $\alpha = 1$, we often omit it, simply writing
 $\operatorname{Lip} (X)$ and $\operatorname{Lip} (f)$.
\end{definition}

While the notation does not make it explicit, since we are considering
functions defined on $(\mathbb{R}^d, \|\cdot\|)$,
$\left\|  M\right\|  _{\operatorname{Op}(\omega)}$ depends
both on $d$ and on  $\|\cdot\|$. But dimension and
norm
will always be clear from context.

\begin{remark}\label{univbds} Kinnunen's Theorem (\ref{kin}) shows that  $\left\|  \mathcal{M}\right\|_{\operatorname{Op}(\omega)} \le 1$  for all sorts of maximal operators $\mathcal{M}$, since for every
$t>0$ and every $f$ with $\left\|  f\right\|  _{\operatorname{Lip}(\omega, X)} \le 1$, we have
$
\omega\left(\mathcal{M}f, t\right)\leq \omega\left(|f|, t
\right) \leq \omega\left( f, t
\right) \le \omega (t)$.
\end{remark}

\begin{example} Note that for $f\in \operatorname{Lip}(\omega)$,
its norm
$\left\|  f\right\|_{\operatorname{Lip}(\omega)}$ depends not only on
the space $\operatorname{Lip}(\omega)$, but also on the modulus
$\omega$ used to define it (we write norm for short, even though we
mean seminorm). Consider, for instance, the
functions on $\mathbb{R}$ that are both bounded and Lipschitz. As a set, this space can be defined via many moduli, for example
$\omega (t) := \min\{t,1\}$ and $\omega^\prime (t) := \min\{t,1/2\}$.
Then $\psi (x) = (1 - |x|)^+ = \max\{0, 1 - |x|\}$ has norm
1 in $\operatorname{Lip}(\omega)$ and norm 2 in $\operatorname{Lip}(\omega^\prime)$.
We also mention that since $ \lim_{t \to\infty}
\omega\left(\mathcal{M}\psi, t\right) = 1$, we have $\left\|  M\right\|_{\operatorname{Op}(\omega)} = 1$, so the upper bound
1 can be attained by the uncentered maximal operator. But we shall
see that for many standard moduli strict inequality holds.
\end{example}

\begin{remark}\label{univbds}  Identifying some extremal
functions, as we do in Theorem \ref{mainabst}, allows us to immediately improve the general bound 1 in the
uncentered case, on Lipschitz functions, for every dimension $d$,
and all norms.
\end{remark}

\begin{corollary}\label{univLipbounds} Fix $d\ge 1$.
Given any norm $\|\cdot\|$ on $\mathbb{R}^d$, the
associated maximal operator acting on Lipschitz functions
satisfies
$\|M\|_{\operatorname{Op}(1)}\le d/(d+1).$
\end{corollary}
\begin{proof} Let $\psi (x) =\max\{1 -
\|x\|, 0\}$ (since we are in the Lipschitz case, it is immaterial
which height $\Omega$ we select in Theorem \ref{mainabst}, so we just pick $\Omega = 1$).
Given a vector $v$ with $\|v\|=1$, we estimate $M\psi (v)$
by integrating over the support of $\psi$, that is, over the unit
ball according to $\|\cdot\|$. This gives a
 lower bound for $M\psi $ on the unit sphere, and hence an upper bound for $\|M\|_{\operatorname{Op}(1)}$. Since for  $d \ge 1$ the average of the cone over the unit ball is $1/(d+1)$, we have
$\|M\|_{\operatorname{Op}(1)}\le 1 - 1/(d+1).$
\end{proof}

Of course, the very general but otherwise rather crude estimates given by the preceding corollary
cannot be expected to be sharp. In dimension one there is essentially one norm, and the best constant is  $\sqrt2 -1$, as will be seen
below, rather than  $1/2$. We shall show that in dimensions two and three
the constants $2/3$ and $3/4$ can be improved when dealing with the
$\ell_1$ and $\ell_\infty$ norms.

Let $f$ be a smooth, compactly supported
function
on $\mathbb{R}^d$, and recall that $\operatorname{Lip}(Mf) = \|DMf\|_\infty$. Since by the previous corollary $\|DMf\|_\infty\le
d/(d+1) \|Df\|_\infty$, it is
natural to suspect that if $p_d$ is high enough, for every $p\ge p_d$
one can find a $c_p\in (0,1)$ such that $\|DMf\|_p\le c_p \|Df\|_p$ (with $c_p$ independent of $f$, see also
question 2 below). But in this paper we study the size
of $\|DMf\|_p$ only when
 $p=\infty$.

Note also that asymptotically the above corollary does not improve  the general upper bound  $1$. Thus, it is natural to enquire whether bounds
strictly less than 1 and independent of the dimension
can be obtained,
 by a more careful choice of averaging
ball. It turns out that  for cubes (with sides parallel to the axes, that is, $\ell_\infty$ balls) and
cross-polytopes ($\ell_1$ balls) the constant $1$ is the correct
asymptotic value, and the trivial upper bound $d/(d+1)$ from the previous result is ``essentially" optimal (actually, we shall see later that the lower bound $(d-1)/(d+1)$ is a better asymptotic estimate).
However, substantial improvement is possible for Euclidean ($\ell_2$) balls, and thus, the same question  on the $p$ norm of the derivative arises: Can we have  $c_p < 1$ for $p <\infty$ sufficiently  high?
And can we take $p$ to be independent of the dimension?

We conclude this section by using the notation from Definition \ref{lipopnorm} to summarize the main contents of Theorem
\ref{mainabst}. While formula (\ref{equalityopnorm}) below does not look very promising, due to the successive appearance of two suprema and one infimum, the fact is that it will allow us to obtain optimal asymptotic
estimates for the norms considered in the next section, and in
some cases we will be able  to actually {\em compute} the number $\left\|  M\right\|_{\operatorname{Op}(\omega)}$.

\begin{corollary}\label{opnormequa}
Let  $\|\cdot\|$ be a norm on $\mathbb{R}^d$, and let $\omega$
be a  modulus
 of continuity. Then the
associated maximal operator $M$
acting on $\operatorname{Lip}(\omega)$ has  norm given by
\begin{equation}\label{equalityopnorm}
\left\|  M\right\|_{\operatorname{Op}(\omega)}
  =
\sup_{t>0}\sup_{\left\{v\in\mathbb{R}^{d}:\;\left\|  v\right\|  = 1\right\}}
\inf_{\left\{c\in\mathbb{R}^{d}, 0 < R\le 1 :\left\|  v-c\right\| =
R\right\}}\frac{1}{\omega(t)}
{\displaystyle\int_{B(0,1)}} \omega\left( t\left\|
c+Ru\right\|  \right) \frac{du}{\left|  B(0,1)\right|}.
\end{equation}
\end{corollary}
\begin{proof} Using one of the extremal functions $\psi$  found
in Theorem \ref{mainabst} we have
$\left\|  M\right\|_{\operatorname{Op}(\omega)} = \left\|  M \psi\right\|_{\operatorname{Lip}(\omega)}$, and now the result follows
from (\ref{equality}) together with (\ref{infset2}).
\end{proof}

\section
{Bounds in the $\ell_\infty$, $\ell_2$ and $\ell_1$ cases.}

Next we specialize Theorem \ref{mainabst} and Corollary
\ref{opnormequa} to the three norms in the title of this section.
This specialization will allow us to find explicit constants,
at least for low dimensions and in the case of cubes (where the needed
arguments seem to be simpler). As we mentioned, however, all
constants will turn out to be exactly the same when working
with the $\ell_\infty$  and $\ell_1$ norms.

\begin{theorem}\label{cubes} Let $M$ be the uncentered maximal operator
associated to balls defined by the
$\ell_\infty$ norm on $\mathbb{R}^{d}$, i.e., to cubes with sides
parallel to the coordinate axes, and let $\omega$ be a modulus
of continuity. Then
\begin{equation}\label{dcubes}
\left\|  M\right\|  _{\operatorname{Op}(\omega)}= \sup_{t>0} \left\{
\inf_{0\leq s\leq1} \dfrac{1}{\omega\left( t\right)}
{\displaystyle\int_{[-s,1]^d}}\omega\left(  t\left\|
x\right\|_{\infty} \right)  \frac{dx}{\left(1+s\right)^{d}}\right\},
\end{equation}
or equivalently,
\begin{equation}\label{dcubesbis}
\left\|  M\right\|  _{\operatorname{Op}(\omega)}= \sup_{t>0} \left\{
\inf_{0\leq s\leq1} \dfrac{d}{(1+s)^d\omega\left( t\right)}\left[
{\displaystyle 2^d \int_0^s u^{d-1}}\omega\left(  t u\right)
du+\int_s^1(u+s)^{d-1}\omega(tu)du\right]\right\}.
\end{equation}
Moreover, if we choose the same modulus $\omega$ for every dimension $d$,
then $\left\| M\right\|  _{\operatorname{Op}(\omega)}$ is
nondecreasing in $d$,
 and
$\lim_{d\to\infty}\left\|  M\right\|
_{\operatorname{Op}(\omega)}=1$.
\end{theorem}

Of the two expressions for $\left\|  M\right\|  _{\operatorname{Op}(\omega)}$ given in the preceding
theorem, (\ref{dcubes}) is the one with  a clearer
 geometric content: It says that an optimal choice for
 $v$ is $(1,1, \dots,1)$, and the associated minimizing cube
 contains the origin and is contained in $[-1,1]^d$.
On the other hand, while harder to interpret, formula  (\ref{dcubesbis}) turns out to be
computationally much more convenient.

\begin{proof} From (\ref{equalityopnorm}) it follows that
\begin{equation}\label{mincube}
\left\|  M\right\|  _{\operatorname{Op}(\omega)}= \sup_{t>0}\sup_{\left\{  v\in\mathbb{R}^{d}
:\left\|  v\right\|_\infty  =1\right\}  } \inf_{\left\{
c\in\mathbb{R}^{d},R>0:\left\|  v-c\right\|_\infty
= R\le 1\right\}}\dfrac{1}{\omega\left( t\right)} {\displaystyle\int_{[-1,1]^d}}
\omega\left(  t\left\|  c+Ru\right\|_\infty  \right)  \dfrac{du}{2^d},
\end{equation}
so if we write $v_0:=(1,1,\ldots,1)$, then
\begin{equation}\label{mincubelo}
\left\|  M\right\|  _{\operatorname{Op}(\omega)}\ge \sup_{t>0} \inf_{\left\{
c\in\mathbb{R}^{d},R>0:\left\|  v_0-c\right\|_\infty
= R\le 1\right\}}\dfrac{1}{\omega\left( t\right)} {\displaystyle\int_{[-1,1]^d}}
\omega\left(  t\left\|  c+Ru\right\|_\infty  \right)  \dfrac{du}{2^d}.
\end{equation}
We claim that the infimum in (\ref{mincubelo}) is attained when
$c=(1-R,1-R,\ldots,1-R)$ for some $R \in (0,1]$ (which may vary with
 the dimension).  Let $c=(c_1,c_2,\ldots,c_d)$ satisfy
$\|v_0-c\|_\infty = R\le1$. Then $c_1,c_2,\ldots,c_d\ge 1-R$.
Suppose $c_1 > 1-R$ (else, do nothing and consider $c_2$ instead).
Translate $c$ to $c^1:=(1-R,c_2,\ldots,c_d)$, by moving it parallel
to $e_1$. Now, for every $x\in B(c,R)\setminus B(c^1,R)$,
$|x_1|>1$, while for every $x\in B(c^1,R)\setminus B(c,R)$,
$|x_1|\le 1$. Since parallel transport in the direction of $e_1$
does not change any of the other coordinates, the average value of
$\|\cdot\|_\infty$ is not increased. Then repeat this process with each
coordinate.

Next, note that  the infimum is attained when $1/2\le R\le1$.
This immediately follows from the fact that  for each
$u\in B(0,1)$, the function $f(R):=\|(1-R)v_0+Ru\|_\infty$ is
decreasing on $(0, 1/2]$, since
each coordinate function $|1-R+Ru_i| = 1-(1- u_i)R$ is decreasing there. Thus
\begin{equation}\label{mincubeloR}
\left\|  M\right\|  _{\operatorname{Op}(\omega)}\ge \sup_{t>0} \inf_{\left\{
1/2\le  R\le 1\right\}}\dfrac{1}{\omega\left( t\right)} {\displaystyle\int_{[-1,1]^d}}
\omega\left(  t\left\|  (1-R) v_0 +Ru\right\|_\infty  \right)  \dfrac{du}{2^d}.
\end{equation}
On the other hand, using (\ref{mincube}) and, for each
each unit vector $v$, taking
the infimum over a smaller set of associated centers $c$, we get
\begin{equation}\label{mincubehi}
\left\|  M\right\|  _{\operatorname{Op}(\omega)}\le \sup_{t>0}\sup_{\left\{  v\in\mathbb{R}^{d}
:\left\|  v\right\|_\infty  =1\right\}  } \inf_{\left\{
1/2\le  R\le 1\right\}}\dfrac{1}{\omega\left( t\right)} {\displaystyle\int_{[-1,1]^d}}
\omega\left(  t\left\|  (1-R) v +Ru\right\|_\infty  \right)  \dfrac{du}{2^d}.
\end{equation}
We show that the supremum over unit vectors is attained on $v_0$, so in fact the right hand sides of (\ref{mincubeloR}) and
(\ref{mincubehi}) are equal. From this, (\ref{dcubes}) follows by making the change of variable  $x= (1-R)v_0+Ru$, and relabeling $s=2R -1$.
Fix
$R\in [1/2, 1]$, and  let $v = (v_1,\ldots,v_d)$ satisfy $\|v\|_\infty=1$. By symmetry considerations we may assume
that $v_1,\ldots,v_d \ge 0$. Now we argue as before. If $v_1 =1$ do
nothing and move to $v_2$. Else, $v_1 < 1$, so shift $v$ to
$v^1:=(1,v_2,\ldots,v_d)$ by parallel transport in the direction of
$e_1$. Then for every $x\in B((1-R)v,R)\setminus B((1-R)v^1,R)$,
$(1-R)v_1-R<x_1<0$, so $|x_1|<|(1-R)v_1-R|$, while if $x\in
B((1-R)v^1,R)\setminus B((1-R)v,R)$, we have $|x_1|>(1-R)v_1+R$.
Since the average value of $|x_1|$ increases after the shift and the other
coordinates do not change, the average value of $\|\cdot\|_\infty$
increases. Then repeat this process with each
coordinate.

To obtain (\ref{dcubesbis}) from (\ref{dcubes}), break up the integral
appearing in (\ref{dcubes}) into the regions $\{\|x\|_\infty < s\}$
 and $\{\|x\|_\infty \ge s\}$, and then separate these
 into the sets where $|x_i|=\|x\|_\infty$, for
$i=1,\ldots,d$. Note, for instance, that if we are working over
$\{\|x\|_\infty \ge s\}\cap \{x_1=\|x\|_\infty\}$, then, for a fixed value of $x_1$, the coordinates
$x_2,\dots,x_d$ of the associated vertical section $E_{x_1}$ satisfy
$-s\le x_i\le x_1$, so $|E_{x_1}| = (s + x_1)^{d-1}$. A similar
remark can be made about the integral over $\{\|x\|_\infty < s\}$,
so
 applying Fubini's Theorem we obtain (\ref{dcubesbis}).

Next we show that the norm of $M$ does not decrease when the
dimension changes from $d$ to $d+1$ if we keep the same modulus
$\omega$. For notational simplicity, we suppose that $d=1$. The
argument for arbitrary $d$ is the same. Let $s_d\in [0,1]$ be the
minimizing value of $s$ in dimension $d$. Then
\begin{gather*}
\frac1{ (1+s_2)^{2}}\int_{[-s_2,1]^2}
 \omega( t\max\{  |  x|  ,|  y|\}) dxdy\ge
  \frac1{ (1+s_2)^{2}}\int_{-s_2}^{1}\int_{-s_2}^{1}
 \omega( t|  x| ) dxdy\\
= \frac1{ 1+s_2}\int_{-s_2}^{1}
 \omega( t|  x| ) dx \ge \frac1{1+s_1}\int_{-s_1}^{1} \omega( t|  x| ) dx.
\end{gather*}

To finish, we show that as the dimension $d\to\infty$, for a fixed
modulus $\omega$ we have $\|M\|_{\operatorname{Op}(\omega)}\to 1$.
This is simply a consequence of the fact that for every $s\in [0,1]$, in high dimensions the
measure of the cube $[-s,1]^d$ concentrates near the norm one
vectors. More precisely, let $X$ be a random vector, chosen
uniformly from $[-s,1]^d$. Then its coordinate functions $X_i$ are
independent random variables, uniformly distributed over $[-s, 1]$.
By slight abuse of notation we use $P$ to denote both uniform
probabilities on $ [-s,1]^d$ and on $ [-s,1]$. Fix $0<\varepsilon
<1$, set $t=1$, and choose $\delta \in (0,\varepsilon)$ so that
$\omega(1-\delta)/\omega(1) > 1 -\varepsilon$. Then, for any $0\le
s\le 1$ and all sufficiently high $d$, we have
$$P(\|X\|_\infty > 1-\delta) \ge P(\max\{X_1,\dots,X_d\} > 1-\delta)
= 1 - \Pi_{1}^d  P(X_i \le 1-\delta)
$$
$$
= 1 - \left(\frac{ 1+s-\delta}{1+s}\right)^d \ge 1 - \left(1 -
\frac{\delta}{2}\right)^d >1-\varepsilon.
$$
Thus, if the dimension $d$ is large enough (depending of $\omega$ and $\varepsilon$) it follows that
$$
1 \ge \|M\|_{\operatorname{Op(\omega)}} \ge \inf_{0\le s\le
1}\dfrac{1}{\omega\left( 1\right)}
{\displaystyle\int_{[-s,1]^d}}\omega\left(  \left\|
x\right\|_{\infty} \right) \frac{ dx}{\left(1+s\right)^{d}} \ge
(1-\varepsilon) \frac{\omega \left(  1-\delta\right)}{\omega(1)} >
(1-\varepsilon)^2.
$$
\end{proof}

 Observe that in the above proof
we did not need to establish how the optimal value $s_d$ of $s\in [0,1]$
behaves as $d\to\infty$. Intuition suggests that since the measure
of $[-s_d, 1]^d$ concentrates near its border as $d$ grows, in order to minimize
the average value of the norm over this cube, the origin should be increasingly closer to the boundary of $[-s_d, 1]^d$. Or, in other words, $s_d$ should approach $0$ as $d\to\infty$. This intuition is, in fact,
completely erroneous, as formula (\ref{lipanyd}) below shows
(see also the proof of  Corollary \ref{alldinftybounds}): When $d\to\infty$, the optimal $s_d$ tends to 1, and thus the optimal
averaging cube has $\ell_\infty$ diameter approaching 2. Nevertheless, we shall
show that in the euclidean case the above intuition is correct:
As $d\to\infty$
the origin must indeed be increasingly closer to the boundary of
the optimal ball, in order to minimize
the average value of the norm, and the $\ell_2$ diameter
of the optimal averaging ball must approach 1 rather than 2.
This helps to understand why the asymptotic behavior of
$\|M\|_{\operatorname{Op}(1)}$ is so different in the
$\ell_\infty$ and $\ell_2$ cases.

Next we specialize the preceding theorem to the Lipschitz and H\"older functions,
obtaining the following corollary.

\begin{corollary}
  \label{corhold}
  Let  $\alpha\in (0,1]$, and consider the  space
  $(\mathbb{R}^d, \|\cdot\|_\infty)$. Then, on $\operatorname{Lip}_\alpha
  (\mathbb{R}^d, \|\cdot\|_\infty)$,
  \begin{equation}\label{holderanyd}
    \|M\|_{\operatorname{Op}(\alpha)}=\min_{0\le s\le 1}\left\{
\frac{d}{(1+s)^d}\left[\frac{2^d
s^{\alpha+d}}{\alpha+d}+\sum_{j=0}^{d-1}\binom{d-1}{j}\frac{s^{d-1-j}-s^{d+\alpha}}{\alpha+j+1}\right]\right\}.
  \end{equation}
In particular, when $\alpha = 1$, that is, for Lipschitz functions on
$(\mathbb{R}^d, \|\cdot\|_\infty)$ we have
  \begin{equation}\label{lipanyd}
    \|M\|_{\operatorname{Op}(1)}=\frac{d}{d+1}-\frac{1}{d+1}\max_{0<s<1}\left\{s-\frac{2^ds^{d+1}}{(1+s)^d}\right\}.
  \end{equation}
\end{corollary}
\begin{proof} To obtain (\ref{holderanyd}),
use  (\ref{dcubesbis}) in Theorem \ref{cubes} with
$\omega\left(t\right) =t^{\alpha}$, and integrate.
As for (\ref{lipanyd}), it does not seem to be easy to derive
it from (\ref{holderanyd}) by substituting $\alpha = 1$. Instead, use
(\ref{dcubesbis}) again, evaluating the integral $\int_s^1(u+s)^{d-1}udu$ via the change of variable $v = u + s$. This yields
(\ref{lipanyd}) but with $\max_{0\le s \le 1}$. To further refine this expression and obtain  $\max_{0 < s < 1}$, we observe that the maximum is actually achieved at some interior
point of the unit interval. This can be seen by writing $g_d (s) :=s-2^ds^{d+1}/(1+s)^d$ for $s\in[0,1]$,
and noting that  since
$g_d(0)=g_d(1)=0$, and $g_d^{\prime\prime}<0$ on $(0,1)$, the
function is strictly concave there. Thus, it has a unique maximum,
which must occur at some interior point.
\end{proof}

\begin{remark} Specializing formula (\ref{holderanyd}) to $d = 1$ and
$d = 2$ gives the following expressions. We mention that
they can also be obtained easily and directly from
(\ref{dcubes}).

 When $d = 1$, we find that for every   $f\in \operatorname{Lip}_\alpha (\mathbb{R})$ and
every $\alpha\in (0,1]$,
\begin{equation}\label{1d}
\|M\|_{\operatorname{Op}(\alpha)}=\min_{0<s<1}\left\{
\dfrac{1+s^{1+\alpha}}{(1+\alpha)\left( 1+s\right)}\right\}.
\end{equation}
Furthermore, this result is independent of the $\mathbb{R}$ norm, since $d =1$.
The fact that the unique minimum occurs in the interior of $(0,1)$ is
shown, as above,  by elementary calculus arguments:
Fix $\alpha\in (0,1]$, and  for $s\in [0,1]$ write  $g_\alpha (s) :=
\dfrac{1+s^{\alpha+1}}{(\alpha+1)\left(  1+s\right)}$. Evaluating
$g_\alpha$ on 0 and 1 we see (by inspection)  that it achieves its maximum value at
these points. Since $g_\alpha^{\prime\prime} > 0$ on $(0,1)$, the
function is strictly convex there, and thus it has a unique minimum.

It is clear from (\ref{1d}) that on the real line, $\|M\|_{\operatorname{Op}(\alpha)} < (1+\alpha)^{-1}$, which
is the sharp bound in the local case, that is, for H\"older
functions on intervals (cf. Corollary \ref{lipholloc} below).
Thus, H\"older constants are smaller on the line than on proper
subintervals. Note also that as $\alpha
\to 0$, $\|M\|_{\operatorname{Op}(\alpha)} \to 1$, again by
(\ref{1d}). Convergence to 1 as $\alpha
\to 0$ holds also in dimension 2, by (\ref{2d}) below.

When $d=2$, formulas (\ref{holderanyd}) and (\ref{dcubes}) become
\begin{equation}\label{2d}
\|M\|_{\operatorname{Op}(\alpha)} =\min_{0\leq
s\leq1}2\left(\dfrac{\alpha+1+\left( \alpha+2\right) s+\left(
2\alpha+1\right) s^{\alpha+2}}{(\alpha+1)(\alpha+2)\left( 1+s\right)
^{2}}\right).
\end{equation}
Here calculus arguments regarding extrema are more involved, and in fact,
as $d$ grows the formulas given by (\ref{holderanyd}) become less
manageable, though of course,
numerical estimation is possible. However, in the simpler Lipschitz
case, where $\alpha =1$, there is still a good deal of explicit information that can
be extracted from (\ref{lipanyd}). Contrary to our usual notation, in the next corollary
we shall indicate the dependency of the
maximal operator on the dimension $d$ by writing $M_d$.
\end{remark}

\begin{corollary}
  \label{alldinftybounds}
  On $\operatorname{Lip}
  (\mathbb{R}^d, \|\cdot\|_\infty)$,   $\|M_d\|_{\operatorname{Op}(1)} = (d - 1)/(d + 1) +
  o (1/(d + 1))$. More precisely,
    \begin{equation}\label{hilo}
   \frac{d-1}{d+1} <  \|M_d\|_{\operatorname{Op}(1)}
   \le \frac{d}{d+1} - \frac{1}{d+1}\left(1 - \frac{1}{\sqrt d} \right)
   \left[1 - \left(1 - \frac{1}{2\sqrt{d} -1} \right)^d\right].
  \end{equation}
\end{corollary}
\begin{proof} The lower bound follows from (\ref{lipanyd}) by noticing
that   $g_d (s) :=s-2^ds^{d+1}/(1+s)^d < s \le 1$ for all $s\in (0,1]$.
To get an upper bound it is enough to take (\ref{lipanyd}) and assign
any value from $[0,1]$ to $s$. For instance, the choices $s = 0$
and $s = 1$ recover the general bound $d/(d+1)$ from Corollary
\ref{univLipbounds} (so  $\|M_d\|_{\operatorname{Op}(1)} = (d - 1)/(d + 1) +  O(1/(d + 1))$), but of course one can do better: The right hand side of (\ref{hilo})
is obtained by taking $s= 1 -  1/\sqrt d$.
\end{proof}

While sufficient to prove asymptotic equivalence,
the choice $s= 1 -  1/\sqrt d$ made above is somewhat arbitrary
and can easily be improved, at the cost of getting more complicated
upper bounds (so the second inequality in (\ref{hilo})
is also strict). The optimal choice is the unique solution to the polynomial equation
given next.

\begin{lemma}
  \label{polyeq}
The norm of the maximal operator $M$ on $\operatorname{Lip}
  (\mathbb{R}^d, \|\cdot\|_\infty)$ is obtained by evaluating
\begin{equation}\label{optimalfunc}
h_d (s) := \frac{d}{d+1}-\frac{1}{d+1}\left(s-\frac{2^ds^{d+1}}{(1+s)^d}\right)
\end{equation}
 on the unique solution $s_d$ inside $(0,1)$ of the polynomial equation
 \begin{equation}\label{root}
 p_d (s) := 2^ds^{d+1} - (1+s)^{d+1} + 2^d(d+1)s^d=0.
 \end{equation}
\end{lemma}
\begin{proof} Formula  (\ref{optimalfunc}) follows immediately from
(\ref{lipanyd}). Recall from the proof of Corollary \ref{corhold}
that the function
$g_d (s) :=s-2^ds^{d+1}/(1+s)^d$ has a unique maximum on $(0,1)$, to
be found by solving $g_d^{\prime} = 0$, or equivalently,
 $p_d (s) :=(1+s)^{d+1}-2^d(d+1)s^d-2^ds^{d+1}=0$, on the said interval.
 \end{proof}

When $d \le 3$, $\deg p_d \le 4$, so its roots can be found explicitly
using Cardan's formula.
We do this next, thereby obtaining the sharp constant in dimension 1,
and for the $\ell_\infty$ norm, the sharp constants in dimensions 2 and 3. Details
are included for the readers convenience.

\begin{corollary}
  \label{lip1} On $\operatorname{Lip}(\mathbb{R})$ we have
  $\|M\|_{\operatorname{Op}(1)}=\sqrt{2}-1.$
\end{corollary}

\begin{proof} Solving
$p_1 (s)=1 - 2s - 2s^2= 0$ on $(0,1)$, we find that $h_1(s_1) = h_1(\sqrt2
- 1) = \sqrt2 -1$.
\end{proof}

\begin{corollary} \label{lip2}
 On $\operatorname{Lip}(\mathbb{R}^2,\|\cdot\|_\infty)$ we have
  \begin{equation}\label{lipdim2}
  \|M\|_{\operatorname{Op}(1)}= \frac{4}{\sqrt3}\cos\left(\frac{5\pi}{18}\right) +
   \sqrt3\sec\left(\frac{5\pi}{18}\right) - \frac14\sec^2\left(\frac{5\pi}{18}\right) -3.
      \end{equation}
\end{corollary}

\begin{proof}
We use Lemma \ref{polyeq},  finding first
the unique root of $ s^3 + 3 s^2 -s -1/3$ in $(0,1)$. The change of variable
$s\mapsto s-1$ leads to the reduced form $s^3  - 4 s + 8/3.$
Since $(8/3)^2 - 4^4/3^3 = -64/27 < 0$, this is the irreducible
case in Cardan's
formula. Following  Vi\`{e}te we write
$s=\left(  4/\sqrt{3}\right)
y$, to
obtain $4y^{3}-3y=-\sqrt{3}/2=\cos\left(  5\pi/6 + 2 \pi k\right)$,
and now we use the trigonometrical identity $4\cos^{3}\left(
\vartheta\right)  -3\cos\left(  \vartheta\right) = \cos\left(  3\vartheta\right)$ to conclude that
$s=\left(
4/\sqrt{3}\right)  \cos\left(  5\pi/18+2k\pi/3\right)$,  $k=0,1,2$.
Of these three roots only the one corresponding to $k=0$ belongs
to $(1,2)$, so
 $s_2 =
 \frac4{\sqrt3}\cos{\frac{5\pi}{18}}
-1 \in (0,1)$. Finally, evaluating $h_2$
(cf. (\ref{optimalfunc})) on $s_2$
and simplifying once more we find that
$$h_2(s_2) =
\frac{4}{\sqrt3}\cos\left(\frac{5\pi}{18}\right) +
\sqrt3\sec\left(\frac{5\pi}{18}\right) -
\frac14\sec^2\left(\frac{5\pi}{18}\right) -3 .$$
\end{proof}

\begin{remark} Therefore $\|M\|_{\operatorname{Op}(1)}
 \approx
0.574$ on $\operatorname{Lip}(\mathbb{R}^2,\|\cdot\|_\infty)$.
\end{remark}

\begin{remark} In dimension two, the results (\ref{2d}) and (\ref{lipdim2}) hold verbatim
if we use the $\ell_1$ norm instead of the $\ell_\infty$ norm,
since in dimension two $\ell_1$ balls are just rotated cubes. In
fact, this phenomenon repeats itself in every dimension, despite the
fact that the geometry is very different when $d\ge 3$. Thus, the
next result holds also for the $\ell_1$ norm.
\end{remark}

\begin{corollary} \label{lip3}
 On $\operatorname{Lip}(\mathbb{R}^3,\|\cdot\|_\infty)$ we have
  \begin{equation}\label{lipdim3}
  \|M\|_{\operatorname{Op}(1)}= 1-\frac{\frac{2^{9/4}}{((\sqrt{8}+\sqrt{7})^{1/3}+(\sqrt{8}-\sqrt{7})^{1/3})^{3/2}}}{\left(1 + \left(\frac{2^{9/4}}{((\sqrt{8}+\sqrt{7})^{1/3} + (\sqrt{8}-\sqrt{7})^{1/3})^{3/2}}-1 \right)^{1/2} \right)^3}.
      \end{equation}
\end{corollary}

\begin{proof}
By Lemma \ref{polyeq}, it is enough to find the
the unique root $s_3$ of $p_3 (s) = 7s^4+28s^3-6s^2-4s-1=0$ in $(0,1)$
and then evaluate $h_3(s_3)$ (see (\ref{optimalfunc})). Using the change of variable
$s= 1/(2t -1)$, we note that $0 < s < 1$ if and only if $t > 1$,
and $p_3 (s) = 0$ if and only if
$  2t^4-8t+3=0.
$
Now it can be checked by direct substitution that
$$
t_3 = \frac{\left((\sqrt{8}+\sqrt{7})^{1/3}+(\sqrt{8}-\sqrt{7})^{1/3}\right)^{1/2}}{2^{3/4}}\left(1 + \left(\frac{2^{9/4}}{((\sqrt{8}+\sqrt{7})^{1/3} + (\sqrt{8}-\sqrt{7})^{1/3})^{3/2}}-1 \right)^{1/2} \right)
$$
satisfies $2t^4-8t+3=0$, and furthermore, it is the unique $t > 1$
with this property, by the uniqueness of $s_3$ in $(0,1)$.
Rather than  substituting the value of $s_3$ directly in $h_3$, it
is more convenient to simplify  $h_3(s_3)$ first.
Note that
$$h_3(s_3) = \frac{3 + 8 s_3 + 6 s_3^2 + 7 s_3^4}{4 (1 + s_3)^3},
$$
and also
$7s_3^4 = - 28 s_3^3  + 6s_3^2 + 4 s_3 + 1$, since  $p_3(s_3) = 0$.
Eliminating the fourth order term and simplifying we get
$$
\|M\|_{\operatorname{Op}(1)}= h_3(s_3) = 1-8\left(\frac{s_3}{1+s_3}\right)^3.
$$
Using
$s_3= 1/(2t_3 -1)$, the preceding equality becomes
$
\|M\|_{\operatorname{Op}(1)}=   1-t_3^{-3},
$
and now (\ref{lipdim3}) follows by substituting in the numerical value of $t_3$.
\end{proof}
\begin{remark} $\|M\|_{\operatorname{Op}(1)}\approx 0.66155$ on
$\operatorname{Lip}(\mathbb{R}^3,\|\cdot\|_\infty)$.
\end{remark}

Next we study the cases $p=2$ and $p=1$.
 Since in one dimension all the $\ell_p$ unit balls coincide with
the interval $[-1,1]$, the case $d=1$ is covered by Theorem
\ref{cubes}.

\begin{lemma}\label{fijainfimo}
 Let $d\ge 2$, let $f: [0,\infty)\to [0,\infty)$ be an increasing function, and let $1\le p< \infty$. Then, for all $R>0$, all $c=(c_1,\ldots,c_d)\in \mathbb{R}^d$, and all $a\in\mathbb{R}$ such that $|a|\le|c_1|$, we have
\begin{equation*}
  \int_{B(ae_1,R)}f(\|x\|_p^p)dx \le
  \int_{B(c,R)}f(\|x\|_p^p)dx.
\end{equation*}
\end{lemma}
\begin{proof}
Write $\hat{c}=(c_2,\ldots,c_d)$ and $\hat{x}=(x_2,\ldots,x_d)$. Using Fubini's Theorem and Anderson's Theorem (cf. Remark \ref{anderson})
 in $d$ and $d - 1$ dimensions, we get
  \begin{equation*}
    \int_{B(c,R)}f(\|x\|_p^p)dx=\int_{c_1-R}^{c_1+R}\left(\int_{B(\hat{c},(R^p-|c_1-x_1|^p)^{1/p})}f(|x_1|^p+\|\hat{x}\|_p^p)d\hat{x}\right)dx_1\ge
  \end{equation*}
  \begin{equation*}
    \int_{c_1-R}^{c_1+R}\left(\int_{B(0,(R^p-|c_1-x_1|^p)^{1/p})}f(|x_1|^p+\|\hat{x}\|_p^p)d\hat{x}\right)dx_1=\int_{B(c_1e_1,R)}f(\|x\|_p^p)dx\ge
    \end{equation*}
    \begin{equation*}
    \int_{B(ae_1,R)}f(\|x\|_p^p)dx.
  \end{equation*}
\end{proof}

To apply
(\ref{equalityopnorm}), it is useful to determine  on which unit vectors the supremum is attained. We have seen that one such vector in
the $\ell_\infty$ case is $(1,\dots,1)$. For other
$p$ norms, if $e_1$ can be selected, this usually leads to
simplification  of the formulas. The next lemma reduces the question
of the optimality of $e_1$ to the two dimensional case.

\begin{lemma}\label{optimald} Let $f: [0,\infty) \to [0,\infty)$ be an increasing function, and let $1\le p < \infty$. If for all $c\in \mathbb{R}^2$ and all $R\in (0,1]$ we have
\begin{equation*}
  \int_{B(c,R)}f(\|x\|_p^p)dx \le
  \int_{B(\|c\|_p e_1,R)}f(\|x\|_p^p)dx,
\end{equation*}
 then
 for all $d \ge 2$, all $c\in \mathbb{R}^d$, and all $R\in (0,1]$,
\begin{equation*}
  \int_{B(c,R)}f(\|x\|_p^p)dx \le
  \int_{B(\|c\|_p e_1,R)}f(\|x\|_p^p)dx.
\end{equation*}
\end{lemma}
\begin{proof} Assume the result is true for  $d\ge 2$. Let $\overline{c}\in\mathbb{R}^{d+1}$ and let $R > 0$. We write
$\overline{c} = (c, b)$, where $c\in \mathbb{R}^d$,
$\overline{x} = (x, y)$,  $x\in \mathbb{R}^d$, and $x = (x_1,\hat{x}),
\hat{x}=(x_2,\dots, x_d) \in \mathbb{R}^{d-1}$. From Fubini's Theorem,  induction, and the assumption for $d=2$, we get
\begin{equation*}
  \int_{B(\overline{c},R)}f(\|\overline{x}\|_p^p)d\overline{x} =
\int_{b-R}^{b+R} \int_{|x_1-c_1|^p+\ldots
  +|x_{d}-c_{d}|^p\le R^p-|y-b|^p} f(\|x\|_p^p + |y|^p) dx dy
\end{equation*}
\begin{equation*}
 \le \int_{b-R}^{b+R} \int_{|x_1-\|c\|_p|^p+ |x_2|^p+\ldots
  +|x_{d}|^p\le R^p-|y-b|^p} f(\|x\|_p^p + |y|^p) dx dy
  \end{equation*}
\begin{equation*}
=  \int_{\|\hat{x}\|_p^p\le R^p}
\left(\int_{|x_1-\|c\|_p|^p+ |y-b|^p \le  R^p -\|\hat{x}\|_p^p} f(\|\hat{x}\|_p^p + |x_1|^p+ |y|^p) dx_1 dy\right) d\hat{x}
\end{equation*}
\begin{equation*}
\le  \int_{\|\hat{x}\|_p^p\le R^p}
\left(\int_{|x_1-\|\overline{c}\|_p|^p+ |y|^p \le  R^p -\|\hat{x}\|_p^p} f(\|\hat{x}\|_p^p + |x_1|^p+ |y|^p) dx_1 dy\right) d\hat{x}
\end{equation*}
\begin{equation*}
 =  \int_{B(\|\overline{c}\|_p e_1,R)}f(\|\overline{x}\|_p^p)d\overline{x}.
 \end{equation*}
\end{proof}

 For euclidean balls, optimality of $e_1$, or any other vector (in every dimension), follows
from symmetry. Now let $d = 2$. For $\ell_1$ balls, optimality of $e_1$ follows from
the optimality of $(1,1)$ for $\ell_\infty$ balls, since each unit ball and its corresponding norm can be obtained from the other via a rotation
and a dilation. And in this case $e_1$ is strictly better than nearby
vectors. We have not been able to prove in a direct way
the optimality of $e_1$ for other $\ell_p$
balls,  $1 < p < 2$,  even for explicit moduli (approximation
arguments seem to yield very limited results).

It is well known, and we use it below in the cases $p=1,2$, that when
$1\le p <\infty$ the measure of the unit ball concentrates near the
vectors with norm one and first coordinate equal to zero, that is,
near the ``vertical equator" perpendicular to $e_1$:
$\{x\in\mathbb{R}^d: x_1 = 0 \mbox{ and }\|x\|_p =1\}$. Indeed,
 the sections of the unit ball perpendicular to $e_1$ are
balls in $\mathbb{R}^{d-1}$, and a small decrease in the radius $r$
causes a large decrease in mass whenever $d$ is high, since
 Lebesgue measure in $\mathbb{R}^{d-1}$ scales
like $r^{d-1}$. It follows that the measure of the unit ball
concentrates on the sections of maximal radius, i.e., when
$x_1\approx 0$.
Likewise, the measure of the unit ball concentrates near the unit sphere.
Thus we have concentration near the vertical equator,
since the intersection of two very large subsets of the unit ball
must be large. Of course,  when
$p=\infty$ this concentration near the vertical equator
does not takes place.

Let $\mathbb{B}_p^d$ be the
$\ell_p$ unit ball
$\{x\in\mathbb{R}^{d}:\;\left\|x\right\|_p\leq1\}$ (for
convenience here we take balls to be closed), let $\mathbb{S}_p^{d-1}$ be
the corresponding unit sphere, given by
$\{x\in\mathbb{R}^{d}:\;\left\|x\right\|_p=1\}$, and let
$|\mathbb{B}_p^{d}|$ and $|\mathbb{S}_p^{d-1}|$ be their respective $d$
and $d-1$ volumes. When considering radii $r$
not necessarily equal to one, we write $\mathbb{B}_p^d(r)$ and $\mathbb{S}_p^{d-1}(r)$.

\begin{theorem}\label{2balls} Let $d\ge 2$, let $M$ be the uncentered maximal operator associated to balls defined by the $\ell_2$ norm,
i.e., to euclidean balls, and let $\omega$
 be a
 modulus of continuity. Then
\begin{equation}\label{2ballsformM}
\left\|  M\right\|
_{\operatorname{Op}(\omega)} =   \sup_{t>0}   \inf_{1/2\le
R\leq1}\frac{(d-1)\Gamma (1 + d/2)}{\omega\left( t\right)
\sqrt{\pi} \Gamma (1/2 + d/2)}
\times
\end{equation}
\begin{equation*}
 {\displaystyle\int_{-1}^{1}}
{\displaystyle\int_{0}^{(1-u_1^2)^{1/2}}}
\omega\left(t((1 - R + R u_1)^2+ R^2\rho^2)^{1/2}\right)
\rho^{d-2}d\rho du_1.
\end{equation*}
Furthermore, if we select the same modulus $\omega$ in all dimensions,
we have
\begin{equation}\label{limit2lo}
\sup_{t>0}\left\{\frac{\omega\left(
2^{-\frac{1}{2}} t\right)}{\omega\left( t\right)} \right\} \le \liminf_{d\to\infty} \left\|  M\right\|
_{\operatorname{Op}(\omega)}
\end{equation}
and
\begin{equation}\label{limit2hi}
 \limsup_{d\to\infty} \left\|  M\right\|
_{\operatorname{Op}(\omega)} \le
\inf_{r > 1}\sup_{t>0}\left\{\frac{\omega\left(
2^{-\frac{1}{2}} r t\right)}{\omega\left( t\right)} \right\}.
\end{equation}
Under the additional assumption that $\omega$ is concave, the
 limit exists, it is equal to the left hand side of (\ref{limit2lo}), and bounds $\left\|  M\right\|
_{\operatorname{Op}(\omega)}$ uniformly in $d$: Given $d \ge 1$,
\begin{equation}\label{upper2}
\left\|  M\right\|
_{\operatorname{Op}(\omega)}\le
\sup_{t>0}\left\{\frac{\omega\left(
2^{-\frac{1}{2}} t\right)}{\omega\left( t\right)} \right\}
=  \lim_{d\to\infty} \left\|  M\right\|
_{\operatorname{Op}(\omega)}.
\end{equation}
In particular, for the H\"older and Lipschitz classes we obtain
$$\left\|  M\right\|
_{\operatorname{Op}(\alpha)}\le 2^{-\frac{\alpha}{2}} \mbox{ \ \ \ for all }d\ge 1,
\mbox{  and \ \ \ }\lim_{d\to\infty}\left\|  M\right\|
_{\operatorname{Op}(\alpha)} = 2^{-\frac{\alpha}{2}}.$$
\end{theorem}

In order to give an idea about the size of the constant term (for $d$ fixed)
in (\ref{2ballsformM}), we point out that
\begin{equation}\label{ratio}
\left(\frac{d}{2}\right)^{1/2}
\le \frac{\Gamma (1 + d/2)}{\Gamma (1/2 + d/2)}
\le \left(\frac{d+1}{2}\right)^{1/2}.
\end{equation}
This is an easy consequence of the log-convexity of the $\Gamma$
function (cf. Exercise 5, pg. 216 of \cite{Web}). Note also
that if instead of taking the infimum in the right hand side
of (\ref{2ballsformM}) we just set $R=1$, we are averaging over
the whole unit ball, that is, we are acting as in the proof of
Corollary \ref{univLipbounds}. And indeed, if we change to polar
coordinates and integrate we recover the bound $d/(d+1)$. So here, a better choice of $R$ leads to lower asymptotic bounds. We shall see
that in fact, when $d\to\infty$ the optimal $R= R(d)$ approaches $1/2$, as intuition
suggests. The $s$-dimensional Hausdorff measure on $\mathbb{R}^d$
is denoted by $\mathcal{H}^s$.

\begin{proof} As in the case of Theorem \ref{cubes},
 it is enough to consider
balls contained in $\mathbb{B}_2^d$. To see this, fix any unit
vector $v$, and let $B(c,R)$ be a minimizing ball for $v$ in
(\ref{opnormequa}). Suppose $B(c,R)$ has points outside
$\mathbb{B}_2^d$. Translating $B(c,R)$ towards the origin along the
ray $\{t c: t \ge 0\}$ determined by the vector $c$, so that the displaced ball
$B(c^\prime, R)$ is fully contained in $\mathbb{B}_2^d$ and tangent
to the unit sphere, leads to
$$ {\displaystyle\int_{B(0,1)}}
\omega\left(  t\left\|c^\prime + Ru\right\|_2  \right)  du \le {\displaystyle\int_{B(0,1)}}
\omega\left(  t\left\|  c+Ru\right\|_2  \right)  du$$
by Anderson's Theorem (see Remark \ref{anderson}). It may well happen that after the
translation $v \notin B(c^\prime, R)$. If so, rotate
$B(c^\prime, R)$ about the origin to make its new center lie in
the segment
$[0,v]$. Since this does not change the value of the integral,
we conclude that  it suffices to consider
balls contained in $\mathbb{B}_2^d$. Again by rotational
symmetry we may take $v$ to be $e_1$, so it is enough to consider
centers $c= (1 - R) e_1$  and radii $R$, with $0\le R\le1$.
Hence, by (\ref{equalityopnorm})  we have
\begin{equation}\label{defina}
\left\|  M\right\|
_{\operatorname{Op}(\omega)} =   \sup_{t>0}   \inf_{0\le
R\leq1} \frac{1}{\omega(t)}
{\displaystyle\int_{\mathbb{B}_2^d}} \omega\left(t\left\|
(1-R) e_1+ R u\right\|_2  \right)  \frac{du}{\left|  \mathbb{B}_2^d\right|}.
\end{equation}
 Note that the infimum in  (\ref{defina}) is attained when $R\in[1/2, 1]$. In fact, we claim that
the function $f_2 (R) := \left\|(1-R) e_1+ R u\right\|_2^2$ is decreasing
when $R\leq1/2$, so  the
minimum must indeed occur on $1/2\le R\le 1$. Differentiating
$f_2(R) = (1-R + R u_1)^2  + R^{2} \sum_{i=1}^d u_i^2  - R^2 u_1^2$,
and using $\sum_{i=1}^d u_i^2\le 1$ together with $R \le 1/2 $,
the claim follows.

Given $u\in \mathbb{B}_2^d$, we write
$u =(u_1, y)$, where $u_1\in\mathbb{R}$ and $y\in\mathbb{R}^{d-1}$.
Denote by $P_d$ the uniform probability on $\mathbb{B}_2^d$, so
$d P_d(u) = du/|\mathbb{B}_2^d|$.
To prove (\ref{limit2hi}) and (\ref{limit2lo}), fix $t > 0$, and note that by concentration of
measure near the vertical equator, for each
$\delta >0$  we have
\begin{equation}\label{2conc}
\lim_{d\to\infty} P_d(\{|u_1| <\delta, 1-\delta < \|y\|_2 \le 1\})=1.
\end{equation}
Of course, the weaker assertion
\begin{equation}\label{3conc}
\lim_{d\to\infty} P_d(\{|u_1| <\delta\})=1
\end{equation}
also holds.
Since  $\left\|
(1-R) e_1+ R u\right\|_2 \le 1$ for all $u\in \mathbb{B}_2^d$ and
all  $R\in [1/2,1]$,
\begin{equation}\label{leftover}
  \inf_{1/2\le
R\leq1}
{\displaystyle\int_{\mathbb{B}_2^d\setminus \{|u_1|<\delta\}}} \frac{\omega\left(t\left\|
(1-R) e_1+ R u\right\|_2  \right) }{\omega(t)}
 d P_d(u) \le P_d(\{|u_1| \ge\delta\}).
\end{equation}
Next, note that on
$\mathbb{B}_2^{d}\cap \{|u_1| <\delta\}$, for every $d$ and
every $R\in [1/2, 1]$ we have
\begin{equation}\label{easyupbd}
 \left\|
(1-R) e_1 + R u\right\|_2  \le \left(
(1 - R + \delta)^2 + R^2 \right)^{1/2}.
\end{equation}
The unique minimum of $h_1 (R) := (1 - R + \delta)^2 + R^2$ on $[1/2, 1]$  is attained at
$R_\delta = (1 + \delta)/2$, and there $h_1 (R_\delta) = (1 + \delta)^2/2$. Thus, \begin{equation}\label{easyupbd2}
 \left\|
(1-R_\delta ) e_1 + R_\delta u\right\|_2  \le 2^{-1/2}(1 + \delta),
\end{equation}
so splitting $\mathbb{B}_2^{d}$ into the regions where $|u_1| <\delta$
and $|u_1| \ge \delta$ we obtain
\begin{equation}\label{andright}
\inf_{1/2 \le R \le 1} \int_{\mathbb{B}_2^d} \frac{\omega\left(t\left\|
(1-R) e_1 + R u\right\|_2  \right)}{\omega(t)}   dP_2(u)
\le  \frac{\omega\left( t\left( 2^{-1/2} (1 +\delta)\right) \right)}{\omega (t)}
+P_d(\{|u_1| \ge\delta\}).
\end{equation}
Taking on both sides
of the preceding inequality first $\sup_{t>0}$, second, $\limsup_{d\to \infty}$, and third, $\inf_{\delta >0}$, from (\ref{defina}) (with $R\in[1/2,1]$) we get
\begin{equation}\label{andright3}
\limsup_{d\to\infty} \left\|  M\right\|
_{\operatorname{Op}(\omega)}
\le  \inf_{\delta >0} \sup_{t > 0} \frac{\omega\left(
2^{-\frac{1}{2}} (1 + \delta) t\right)}{\omega\left( t\right)}.
\end{equation}
This proves (\ref{limit2hi}).

 The argument used to obtain (\ref{limit2lo}) is similar. Fix
$u = (u_1,y)\in\mathbb{R}^d$ and note that on $\mathbb{R}$ the
function
$$h_2(R) := (1-R +  R u_1 )^2 + R^2 \|y\|_2^2
= \left\| (1-R) e_1 + R u\right\|_2^2
$$
achieves its  unique minimum at $R_u=(1-u_1)/[(1-u_1)^2+\|y\|_2^2]$,
where it takes the value
$h_2(R_u) = \|y
\|_2^2/[(1-u_1)^2+\|y\|_2^2]$. Thus, for $0 < \delta < 1$ and  $u\in
\mathbb{B}_2^{d}\cap \{|u_1| <\delta, 1-\delta < \|y\|_2 \le 1\}$
 we have
\begin{equation*}
  \frac{1-\delta}{\sqrt{(1+\delta)^2+ 1}}\le
  \frac{\|y\|_2}{\sqrt{(1-u_1)^2+\|y\|_2^2}}\le\|(1-R)e_1+Ru\|_2,
\end{equation*}
from which it follows that
\begin{equation}\label{left1}
\sup_{t > 0}  P_d(\{|u_1| <\delta, 1-\delta < \|y\|_2 \le 1\})
\frac{\omega\left( t\left((1+\delta)^2+1\right)^{-1/2}
(1-\delta)\right)}{\omega(t)}
\end{equation}
\begin{equation}\label{left11}
\le \sup_{t > 0}\inf_{1/2 \le R \le 1} \int_{\mathbb{B}_2^d}
\frac{\omega\left(t\left\| (1 - R) e_1 + R u\right\|_2
\right)}{\omega(t)}   dP_d(u) = \left\|  M\right\|
_{\operatorname{Op}(\omega)}.
\end{equation}
Now (\ref{left1}) increases with $d$, so taking the limit
inferior in (\ref{left1}) and (\ref{left11})
as $d\to\infty$, we get
\begin{equation}\label{left2}
\sup_{t > 0} \frac{\omega\left(  t\left((1+\delta)^2+1\right)^{-1/2}
(1-\delta) \right)}{\omega(t)} \le \liminf_{d\to\infty} \left\|
M\right\| _{\operatorname{Op}(\omega)}.
\end{equation}
Likewise, the left hand side of the preceding inequality is
decreasing in $\delta \in (0, 1)$, so (\ref{limit2lo}) follows by
 taking the supremum over $\delta$, interchanging it with the supremum
 over $t$, and letting $\delta\downarrow 0$.

  Next, write $u = (u_1,y)\in\mathbb{R}^d$ and $y = \rho\eta$, where
$\rho = \|y\|_2$ and $\|\eta\|_2 = 1$. By Fubini's Theorem
\begin{equation}\label{fubini}
\int_{\mathbb{B}_2^d} \omega\left(t\left\|
(1 - R)  e_1 + R u\right\|_2  \right)   du =
\int_{-1}^1\int_{\mathbb{B}_2^{d-1}((1-|u_1|^2)^{1/2})} \omega\left(t\left\|
(1 - R) e_1 + R  (u_1, y)\right\|_2  \right)  d y du_1.
\end{equation}
Using polar coordinates on
the vertical sections, or equivalently, by the coarea formula,  we get
\begin{equation*}
\int_{\mathbb{B}_2^{d-1}((1-u_1^2)^{1/2})} \omega\left(t\left\| (1 -
R) e_1 + R (u_1, y)\right\|_2  \right)  d y
\end{equation*}
\begin{equation*}
= \int_{0}^{(1-u_1^2)^{1/2}} \int_{\mathbb{S}_2^{d-2}(\rho)}
\omega\left((|(1 - R) +  R u_1|^2+ R^2\|y\|_2^2)^{1/2}\right)
 d{\mathcal{H}^{d-2}(y)} d \rho
\end{equation*}
\begin{equation*}
= (d-1) |\mathbb{B}^{d-1}_2| \int_{0}^{(1-u_1^2)^{1/2}}
\omega\left((|(1 - R) + R u_1|^2+ R^2\rho^2)^{1/2}\right) \rho^{d-2}
d \rho.
\end{equation*}
Since $|\mathbb{B}_2^{d}| =\frac{\pi^{d/2}}{\Gamma (1 + d/2)}$,  (\ref{2ballsformM}) follows from  (\ref{defina}) and the preceding equalities.

Suppose next that $\omega$ is concave. We show that the inequality
in  (\ref{upper2}) holds for every
$d \ge 1$.
When $d=1$ the result follows from (\ref{dcubes}) and
the concavity of $\omega$:
\begin{equation}\label{trivial1d}
\left\|  M\right\|  _{\operatorname{Op}(\omega)}\le \sup_{t>0}  \dfrac{1}{\omega\left(
t\right)} {\displaystyle\int_{0}^1}\omega\left(  t
x \right)  dx
\le \sup_{t>0}  \dfrac{1}{\omega\left(
t\right)} \omega\left( t {\displaystyle\int_{0}^1}
x dx\right)  =  \sup_{t>0}  \dfrac{\omega\left( t/2 \right)}{\omega\left(t\right)}.
\end{equation}
Let $d > 1$ and let $t > 0$ be fixed. If $\omega\left(z\right)$ is concave, then so is
$\omega\left(t z^{1/2}\right)$. Thus, given $x \ge y \ge 0$, we have
 $\omega\left((x-y)^{1/2}\right)
+\omega\left((x+y)^{1/2}\right) \leq 2\omega\left(x^{1/2}\right)$.
We write the integral appearing in (\ref{2ballsformM}), together with
its constant term, as
\begin{equation*} I= {\displaystyle\int_{-1}^{1}}
{\displaystyle\int_{0}^{(1-u_1^2)^{1/2}}}
\omega\left(t((1- R) ^2 + R^2 (u_1^2 + \rho^2) + 2R(1-R) u_1)^{1/2}\right)
d\mu
\end{equation*}
where
$$
d\mu =  \frac{|(d-1)\mathbb{B}_2^{d-1}|}{\left|\mathbb{B}_2^{d}\right| } \rho^{d-2} d\rho  d u_1
$$
defines a probability on $\mathbb{B}^2_2\cap \{\rho \ge 0\}$. Actually
we get a probability $\mu_d$ for each dimension $d$, but
since this is not relevant in the following argument we omit
the reference to $d$ in the notation. Now from concavity and the fact that $u_1^2 + \rho^2\le 1$ we get
\begin{equation*}
 I\le 2 {\displaystyle\int_{0}^{1}}
{\displaystyle\int_{0}^{(1-u_1^2)^{1/2}}}
\omega\left(t((1 - R)^2 + R^2 (u_1^2 + \rho^2))^{1/2}\right)
d\mu \le \omega\left(t((1-R)^2 + R^2)^{1/2}\right) .
\end{equation*}
Using $\inf_{1/2\le R\le 1} \omega\left(t((1-R)^2 + R^2)^{1/2}\right)
=
\omega\left( 2^{-\frac{1}{2}} t\right)$ we obtain (\ref{upper2}).

Regarding the assertion about the H\"older and Lipschitz classes, for $d>1$ the
result follows from  (\ref{upper2}) with $\omega (t) = t^\alpha$.
And for $d=1$, the upper bound is immediate from (\ref{trivial1d}).
\end{proof}

\begin{remark} Consider the Lipschitz functions on $\mathbb{R}^d$.
By the preceding theorem, for every $d$ we have    $\|M\|_{\operatorname{Op(1)}} \le 2^{-1/2}$
on
$\operatorname{Lip}(\mathbb{R}^d,\|\cdot\|_2)$. However,
 on
$\operatorname{Lip}(\mathbb{R}^d,\|\cdot\|_\infty)$,
$\|M\|_{\operatorname{Op(1)}} > (d-1)/(d+1)$, by Corollary \ref{alldinftybounds}. Since $2^{-1/2} < 5/7$,
{\em any} optimal constant in the $\ell_\infty$ case, with $d\ge 6$,
is strictly larger than {\em all} the optimal constants in the Euclidean
case ($d = 1,2,3,\dots$). This illustrates the fact that using
different norms on $\mathbb{R}^d$ may result in obtaining very different best constants.
But the opposite can also happen: Best constants are identical for
the $\ell_1$ and $\ell_\infty$ norms, as we shall prove by showing
that the corresponding integral formulas for $\|M\|_{\operatorname{Op(\omega)}}$ are actually the same.
\end{remark}

Even though we are using different
norms  to define maximal operators and their associated moduli
of continuity, we adopt the convention that the $s$-dimensional
Hausdorff measure $\mathcal{H}^s$ is always defined via the Euclidean $\ell_2$ metric, and normalized by the factor
$\pi^{s/2}/\Gamma(1+s/2)$, so if $s=d$ is a positive natural number, the cube of sidelength
1 has Hausdorff $d$ measure 1. A consistent definition is required in order
to use Fubini's Theorem, or more generally, the coarea formula.
For example, by our convention the length of the polygonal curve $\mathbb{S}_1^{1}$ is $4\sqrt 2$,
and not $8$, which would be obtained if length were computed using
the $\ell_1$ distance
 instead of the Euclidean distance.

\begin{theorem}\label{1balls} Let $d\ge 2$, let $M$ be the uncentered maximal operator associated to balls defined by the $\ell_1$ norm, i.e., to cross-polytopes, and let $\omega$
 be a
 modulus of continuity. Then
\begin{equation}\label{1ballsformM}
\left\|  M\right\|
_{\operatorname{Op}(\omega)}=\sup_{t>0}   \inf_{1/2\le
R\leq 1}\frac{d(d-1)}{2 \omega\left( t\right)}
{\displaystyle\int_{-1}^{1}}
{\displaystyle\int_{0}^{1-|u_1|}}
\omega\left(t(|(1 - R) + R u_1|+ R\rho)\right)
\rho^{d-2}d\rho du_1.
\end{equation}
If we choose the same modulus $\omega$ in all dimensions, then
we have
$\lim_{d\to\infty} \left\|  M\right\|
_{\operatorname{Op}(\omega)} = 1$.
\end{theorem}

\begin{proof}
Setting $c=(1-R)v$ in  (\ref{equalityopnorm}),  we get
\begin{equation*}
 \left\|  M\right\|
_{\operatorname{Op}(\omega)}\le
\sup_{t>0}\sup_{\left\{v\in\mathbb{R}^{d}:\;\left\|  v\right\|_1 =
1\right\}} \inf_{0<R<1}\dfrac{1}{\omega(t)}
{\displaystyle\int_{B(0,1)}} \omega\left( t\left\| (1-R)v+Ru\right\|_1
\right)  \frac{du}{\left| B(0,1)\right|}.
\end{equation*}
Optimality of $e_1$ in dimension 2, follows from the fact that
in this case cross-polytopes are just rotated squares, and for $d > 2$, by Lemma
\ref{optimald}. Thus
\begin{equation}\label{cotasup}
  \left\|  M\right\|
_{\operatorname{Op}(\omega)}\le
\sup_{t>0}\inf_{0<R<1}\dfrac{1}{\omega(t)}
{\displaystyle\int_{B(0,1)}} \omega\left( t\left\|
(1-R)e_1+Ru\right\|_1 \right)  \frac{du}{\left| B(0,1)\right|}.
\end{equation}
On the other hand, setting $v=e_1$  in  (\ref{equalityopnorm})
yields
\begin{equation}\label{cotainf}
 \left\|  M\right\|
_{\operatorname{Op}(\omega)}\ge \sup_{t>0}
\inf_{\left\{c\in\mathbb{R}^{d},1>R>0 :\left\|  e_1-c\right\|_1 =
R\right\}}\dfrac{1}{\omega(t)}
{\displaystyle\int_{B(0,1)}} \omega\left( t\left\| c+Ru\right\|_1
\right)   \frac{du}{\left| B(0,1)\right|}.
\end{equation}
It follows from Lemma \ref{fijainfimo}  that, given
 $t$ and $R$, the infimum is attained in the last
inequality  when
$c=(1-R)e_1$. Hence
\begin{equation*}
 \left\|  M\right\|
_{\operatorname{Op}(\omega)} \ge
\sup_{t>0}\inf_{0<R<1}\dfrac{1}{\omega(t)}
{\displaystyle\int_{\mathbb{B}_1^d}} \omega\left( t\left\|
(1 - R) e_1+ R u\right\|_1 \right)   \frac{du}{\left| B(0,1)\right|},
\end{equation*}
and thus we have equality.

  Next, write $u = (u_1,y)\in\mathbb{R}^d$, and $y = \rho\eta$, where
$\rho = \|y\|_1$ and $\|\eta\|_1 = 1$.
By Fubini's Theorem
\begin{equation}\label{fubini}
\int_{\mathbb{B}_1^d} \omega\left(t\left\|
(1 - R) e_1 + R u\right\|_1  \right)   du =
\int_{-1}^1\int_{\mathbb{B}_1^{d-1}(1-|u_1|)} \omega\left(t\left\|
(1 - R) e_1 + R  (u_1, y)\right\|_1  \right)  d y du_1.
\end{equation}
Next we handle the vertical
sections. As a reminder,  we refer the reader to \cite{Fe}, pp. 248-250, or \cite{EG}, pp. 117-119, for basic information on  the coarea formula
\begin{equation}\label{coarea}
\int_{\mathbb{R}^{d-1}} g(y) |Jf(y)| dy =
\int_{\mathbb{R}} \int_{\{f^{-1}(t)\}} g (y)  d{\mathcal{H}^{d-2}(y)} d t.
\end{equation}
Let the Lipschitz function $f(y)$ be $\|y\|_1$.
By definition,
$|Jf(y)|=\sqrt{\operatorname{det} df(y) df(y)^t}$, so for every 
$y$ with no coordinate equal to zero,
a computation shows that $|Jf(y)| = \sqrt{d-1}$. Thus, this is the
case  on almost all
$\mathbb{R}^{d-1}$.  Set
$g(y) = 1/|Jf(y)|$ and use (\ref{coarea}) to obtain\begin{equation*}
|\mathbb{B}^{d-1}_1| = \int_{\mathbb{B}^{d-1}_1}   d y
=
\int_{0}^{1} \int_{\mathbb{S}^{d-2}_1(\rho)} \frac{1}{\sqrt{d-1}} d \mathcal{H}^{d-2}(y) d \rho =
\end{equation*}
\begin{equation*}
=
\frac{|\mathbb{S}^{d-2}_1|}{\sqrt{d-1}}\int_{0}^{1} \rho^{d-2} d \rho
 =   \frac{|\mathbb{S}_1^{d-2}|}{(d-1)\sqrt{d-1}}.
\end{equation*}
From this equality and the coarea formula (once more), with
$$
g(y) = \frac{\omega\left(t(|(1 - R) + R u_1|+R \|y\|_1)\right)}{|Jf(y)|},
$$
we get
\begin{equation*}
\int_{\mathbb{B}_1^{d-1}(1-|u_1|)} \omega\left(t\left\|
(1 - R) e_1 + R  (u_1, y)\right\|_1  \right)  d y
\end{equation*}
\begin{equation*}
=
\int_{0}^{1-|u_1|} \int_{\mathbb{S}_{1}^{d-2}(\rho)}
\omega\left(t(|(1 - R) + R u_1|+ R \rho)\right)\frac{d{\mathcal{H}^{d-2}(y)}}{\sqrt{d-1}} d \rho
\end{equation*}
\begin{equation*}
=
\frac{1}{\sqrt{d-1}}\int_{0}^{1-|u_1|}
\omega\left(t(|(1 - R) + R u_1|+R \rho)\right) \int_{\mathbb{S}_{1}^{d-2}} \rho^{d-2}
d{\mathcal{H}^{d-2}(y)} d \rho
\end{equation*}
\begin{equation*}
=
(d-1) |\mathbb{B}^{d-1}_1| \int_{0}^{1-|u_1|}
\omega\left(t(|(1 - R) +  R u_1|+ R \rho)\right)
\rho^{d-2} d \rho.
\end{equation*}
It is well known (and easy to compute) that $|\mathbb{B}_1^{d}| =\frac{2^{d}}{d!}$,  so $\frac{|\mathbb{B}_1^{d-1}|}{|\mathbb{B}_1^{d}|} =\frac{d}{2}$ and we obtain
\begin{equation*}
\left\|  M\right\|
_{\operatorname{Op}(\omega)}=\sup_{t>0}   \inf_{0\le
R\leq1}\frac{d(d-1)}{2 \omega\left( t\right)}
{\displaystyle\int_{-1}^{1}}
{\displaystyle\int_{0}^{1-|u_1|}}
\omega\left(t(| (1 - R) + R u_1|+ R \rho)\right)
\rho^{d-2}d\rho du_1.
\end{equation*}
 To see that  the infimum  is attained when $R\in[1/2,1]$, we note
 that on $[0, 1/2]$  the function $f_1(R) = 1 -R + R u_1 + R \rho$ has a negative derivative for every $\rho\in [0, 1 - |u_1|)$.
  Finally, $1$ is a uniform upper bound by Kinnunnen's Theorem,
and concentration of measure near the vertical equator $\{u\in\mathbb{R}^d:
u_1 = 0 \text{ and } \|u\|_1 = 1\}$ entails that
$\lim_{d\to\infty} \left\|  M\right\|
_{\operatorname{Op}(\omega)} = 1$.
\end{proof}

Observe, for instance, that unlike the case of cubes and the $\ell_\infty$ norm, it is
not clear from formula (\ref{1ballsformM}) that
$\left\|  M\right\|
_{\operatorname{Op}(\omega)}$ increases with the dimension.
However, this must be the case, since given an arbitrary
modulus,  constants for the $\ell_1$ and the
$\ell_\infty$ norms are equal in each dimension.
Of course this is trivial for $d=1$, and clear for $d= 2$, since
in this case cross-polytopes are  rotated squares. But when
$d > 2$ we have no justification to offer for this phenomenon,
other than the proof below.

\begin{theorem}\label{1ballsareeq} Fix $d\ge 1$ and select a modulus
of continuity $\omega$. Then $\left\|  M\right\|
_{\operatorname{Op}(\omega)}$ has exactly the same value
regardless of whether distances and the maximal operator are computed according to the $\ell_1$ norm, or to the $\ell_\infty$ norm.
\end{theorem}

Recall that in the case of the $\ell_\infty$ norm the
asymptotic value 1 was obtained by using concentration near
the boundary of the cube $[-s,1]^d$, and more precisely,
near the norm 1 vectors, while in the $\ell_1$ case
we utilized concentration near the vertical equator. By the
preceding theorem one of these (different) arguments is redundant.

\begin{proof} Fix $t>0$. To see that the values of
$\left\|  M\right\|
_{\operatorname{Op}(\omega)}$
given by (\ref{dcubes}) and (\ref{1ballsformM})
are indeed the same, it is enough to show that
\begin{equation*}
  \inf_{1/2\le
R\leq1}\frac{d(d-1)}{2} \int_{-1}^{1}
\int_{0}^{1-|u_1|} \omega\left(t(|(1 - R) +  R u_1|+ R \rho)\right)
\rho^{d-2}d\rho du_1
\end{equation*}
\begin{equation}\label{norm1infty}
 = \inf_{0\leq s\leq1} \dfrac{d}{(1+s)^d}\left(
{\displaystyle 2^d \int_0^s z^{d-1}}\omega\left(  t z\right)
dz+\int_s^1(z+s)^{d-1}\omega(tz)dz\right).
\end{equation}
Write
\begin{equation*}
 I := \int_{-1}^{1}
\int_{0}^{1-|u_1|} \omega\left(t(|(1 - R) + R u_1|+ R \rho)\right)
\rho^{d-2}d\rho du_1=I_1+I_2,
\end{equation*}
where
\begin{equation*}
  I_1:=\int_{\frac{R - 1}{R}}^{1}
\int_{0}^{1-|u_1|} \omega\left(t((1 - R) + R (u_1+\rho))\right)
\rho^{d-2}d\rho du_1
\end{equation*}
and
\begin{equation*}
  I_2:=\int_{-1}^{\frac{R - 1}{R}}
\int_{0}^{1+u_1} \omega\left(t(R - 1 + R (\rho-u_1))\right)
\rho^{d-2}d\rho du_1.
\end{equation*}
To compute $I_1$, we use the change of variables
$v= 1 - R + R (u_1+\rho)$, $y= R u_1$ and Fubini's Theorem:
\begin{equation*}
  I_1=\int_{R - 1}^{R }\int_{1 - R +y}^{1-|y|+y}\omega (tv )
  \frac{(v + R -1 -y)^{d-2}}{R^d}dv
  dy
\end{equation*}
\begin{equation*}
=   \frac{1}{R^d}\int_0^1
\int_{\max\{R -1 ,\frac{v-1}{2}\}}^{v + R - 1}\omega (tv )(v + R -1 -y)^{d-2}dy dv
\end{equation*}
\begin{equation*}
  =\frac{1}{R^d(d-1)}\left(\int_{2R - 1}^1
 \omega (tv )\frac{(v+2R - 1)^{d-1}}{2^{d-1}}dv+\int_0^{2R - 1}\omega (tv) v^{d-1}dv\right).
\end{equation*}
Likewise, to compute  $I_2$ we set
$v=  R -1  + R (\rho-u_1)$, $y= R u_1$, and  interchange the order of
integration:
\begin{equation*}
  I_2=\int_{- R }^{R - 1}\int_{R - 1 -y}^{2 R - 1}\omega (tv) \frac{(v+1 - R +y)^{d-2}}{R^d}dv
  dy=\frac{1}{R^d(d-1)}\int_0^{2 R - 1}\omega (tv) v^{d-1}dv.
\end{equation*}
Adding up we obtain
\begin{equation*}
  I=\frac{1}{R^d(d-1)}\left(\int_{2R - 1}^1
  \omega (tv) \frac{(v+ 2 R-1)^{d-1}}{2^{d-1}}dv+2\int_0^{2 R - 1}\omega (tv) v^{d-1}dv\right).
\end{equation*}
To finish,  set $s= 2 R - 1$, multiply by $d(d-1)/2$, and take the corresponding infima to get  (\ref{norm1infty}).
\end{proof}

We conclude this section with four
questions and some variants of these. We conjecture that the following version of Theorems
\ref{2balls} and \ref{1balls} holds for all $p\in (1,2)$
 and $d\ge 2$:
 For every
 modulus of continuity  $\omega$ we have
\begin{equation}\label{pballsformM}
\left\|  M\right\|
_{\operatorname{Op}(\omega)} = \sup_{t>0}   \inf_{0\le
R\leq1}\frac{(d-1)\Gamma (1 + d/p)}{2 \omega\left( t\right)
\Gamma (1 + 1/p) \Gamma (1 + (d-1)/p)}
\times
\end{equation}
\begin{equation*} {\displaystyle\int_{-1}^{1}}
{\displaystyle\int_{0}^{(1-|u_1|^p)^{1/p}}}
\omega\left(t(|(1-R) +  R u_1|^p+R^p\rho^p)^{1/p}\right)
\rho^{d-2}d\rho du_1.
\end{equation*}
 Let $q = p/(p-1)$ be the conjugate exponent
of $p$.  Choose the same modulus of continuity $\omega$ in every dimension $d$.  Then
\begin{equation*}
\liminf_{d\to\infty} \left\|  M\right\|
_{\operatorname{Op}(\omega)} \ge
\sup_{t>0}\left\{\frac{\omega\left(
2^{-\frac{1}{q}} t\right)}{\omega\left( t\right)} \right\}.
\end{equation*}
and
\begin{equation*}
\limsup_{d\to\infty} \left\|  M\right\|
_{\operatorname{Op}(\omega)} \le
\inf_{r>1}\sup_{t>0}\left\{\frac{\omega\left(
2^{-\frac{1}{q}} r t\right)}{\omega\left( t\right)} \right\}.
\end{equation*}
Suppose additionally that   $\omega$ is concave. Then for every $d\ge 1$
\begin{equation}\label{upper}
\left\|  M\right\|
_{\operatorname{Op}(\omega)}\le\sup_{t>0}\left\{\frac{\omega\left(
2^{-\frac{1}{q}} t\right)}{\omega\left( t\right)} \right\}.
\end{equation}
In particular, for the H\"older and Lipschitz classes we obtain
$$\left\|  M\right\|
_{\operatorname{Op}(\alpha)}\le 2^{-\frac{\alpha}{q}} \mbox{ \ \ \ for all }d\ge 1,
\mbox{  and \ \ \ }\lim_{d\to\infty} \left\|  M\right\|
_{\operatorname{Op}(\omega)} = 2^{-\frac{\alpha}{q}}.$$

The main obstacle for proving this result would be removed
by a positive answer to the next question.

\vskip .2cm

{\bf Question 1.} Is $e_1$ a maximizing  vector for $1< p <2$
 when taking the supremum in
 (\ref{equalityopnorm})? Recall that by
 Lemma \ref{optimald} it is enough to consider the case $d=2$.

\vskip .2 cm

 It is
plausible that best bounds increase with the dimension for
 $1\le p <\infty$, as it happens in the case $p=\infty$. This
 would immediately imply that the asymptotic bounds are also
uniform upper bounds.

As we mentioned before, it seems likely that
 $v_p:= d^{-1/p} (1,1, \dots,1)$ is an optimizing
vector when $2 < p < \infty$. But while using $e_1$ leads to simplification of the
integral formulas, using $v_p$ does not. Furthermore,
$e_1$ fits well with Fubini, in the sense that  sections of
$d$-balls yield $d-1$-balls if the sections are perpendicular
to $e_1$. This does not happen with $v_p$, so a change of coordinates would not help on that respect.
In any case, we suspect that
on the range $2 < p < \infty$, as $p\to\infty$ constants become
increasingly worse and approach the bounds that hold for $p=\infty$.

\vskip .2cm

{\bf Question 2.} We have seen that if $f$ is Lipschitz and $M$ is the maximal function associated to euclidean balls, then for every dimension $d$, $\operatorname{Lip}(Mf) \le 2^{-1/2} \operatorname{Lip}(f)$, or equivalently, $\|DMf\|_\infty\le
2^{-1/2} \|Df\|_\infty$. It is natural to
expect a similar behavior for $p$ ``close" to $\infty$. More precisely,
given $c\in ( 2^{-1/2},1)$, is it possible to find a $p_c$ such that
for all $p\ge p_c$, if $Df\in L^p$, then $\|DMf\|_p\le c \|Df\|_p$?
Or given $p >>1$, is it possible to find such a $c\in ( 2^{-1/2},1)$?

\vskip .2cm

Current methods of proof use the $L^p$ inequalities satisfied by $M$ to obtain $W^{1,p}$
results.
Since $\|Mf\|_p\ge \|f\|_p$ always, this type of argument will never
yield constants below 1. On the other hand, if we fix $d$, by Corollary \ref{univLipbounds}
we have $\|DMf\|_\infty\le (1+d)^{-1} d \|Df\|_\infty$ for
the maximal function associated to {\em any} ball, so it is natural
to seek inequalities of the form $\|DMf\|_p\le c_p \|Df\|_p$, where
$c_p < 1$, $p$ is high enough, and $M$ is defined via an arbitrary
norm. At the other extreme, given $f\in W^{1,1}(\mathbb{R}^d)$ with
$d\ge 2$, it is not known whether there is a constant $c_1$ (independent
of $f$) such that $\|DMf\|_1\le c_1 \|f\|_{W^{1,1}(\mathbb{R}^d)}$.
We mention that for some related maximal operators,
such as, for instance, the strong maximal operator (where
averages are taken over rectangles with sides parallel to
the axes) such constants do not exist, cf. Theorem 2.21 of \cite{AlPe2}.
It follows from Theorem 2.5 of \cite{AlPe} that if $d=1$ then
$\|DMf\|_1\le c_1 \|Df\|_1$, and $c_1 = 1$ is sharp. For $d\ge 2$ it
is clear that if any such constant exists, it must be strictly larger
that 1 (take $f$ to be radial, for instance). Whether or not $c_1 <\infty$
for $d\ge 2$, one would expect $c_p > 1$ for small values of $p$, $c_p < 1$ for sufficiently large values of $p$, and $c_p = 1$ for some ``crossing"  $p$, which will likely depend on $d$ and the ball used to define
$M$. A motivation to obtain detailed information on $DMf$ comes not only from the possible use of $DMf$ as a substitute for $Df$, but  also because it will yield new information about the maximal function itself.

\vskip .2cm

On a more speculative mood, we note that where balls of arbitrarily
small radii have to be taken into account, the function is ``large"
and the maximal function coincides with it, so maximal functions
 with smaller Lipschitz norms, and hence lower rate of decay ``from the top", will tend to
be larger in an $L^p$ sense. Since asymptotically
constants are
smaller for $\ell_2$ balls than for cubes or $\ell_1$ balls (and we believe this is also the case for other balls), it is tempting to conjecture that the maximal function associated to
Euclidean balls is at least ``as efficient" at
capturing mass as maximal functions associated to other balls.
Our results suggest this only in the weakest possible sense, since
different norms are used to compute distances,
 the directions we consider when measuring the
modulus of continuity are those of fastest decay and not some ``average direction", and $L^p$ norms of extremal functions
vary with $p$ (at least in the bounded case).
Nevertheless, it seems
worthwhile to try to find out whether the Euclidean maximal function,
both in the centered and uncentered versions,
has larger operator
norm on $L^p$ spaces than maximal functions associated to other balls
(needless to say, weak type results in this line would also be interesting).
In fact, the weaker ``comparison theorem" suggested next would already have many consequences.

\vskip .2cm

{\bf Question 3.} Let $\mathcal{M}_e$ denote the maximal operator, either centered or uncentered,
associated to Euclidean balls, and $\mathcal{M}_b$ the corresponding
(centered or uncentered) operator associated to some other ball (defined by a different norm). Prove or refute the following statement:
For every $p\in (1,\infty)$,
$$10^{78} \|\mathcal{M}_e\|_{L^p(\mathbb{R}^d)\to L^p(\mathbb{R}^d)}
\ge \|\mathcal{M}_b\|_{L^p(\mathbb{R}^d)\to L^p(\mathbb{R}^d)}.$$
Of course, $10^{78}$ is not important here, any other constant $c$ would
do. But we do mean to emphasize that $10^{78}$ does not depend on anything, in particular not on $p$ or $d$. ``Constants" that depend on
the dimension are trivial to obtain by the equivalence of all norms in
$\mathbb{R}^d$. The preceding conjecture, if true in the centered case,
would imply that the uniform bounds in the dimension proved by E. M. Stein for Euclidean balls also hold for all other balls
(including for instance, cubes) and all $p > 1$. An opposing viewpoint can be found in
\cite{Mu}, pg. 298, where
it is suggested
that this result may be false for cubes and $p\le 3/2$. 

The corresponding version
of question 3 for the weak type (1,1) constants is also interesting.
An affirmative answer would entail that the best constants for
the centered maximal function defined using euclidean balls diverge
to infinity with the dimension, since this is the case for cubes,
cf. \cite{Al2}. Thus,  a long standing open
problem by Stein and Str\"omberg  would be solved, cf. \cite{StSt}. 

\vskip .2cm
From the perspective of comparing the sizes of the maximal operators
associated to euclidean balls and to cubes, it would probably
be more telling if Lipschitz or H\"older constants were
computed using the same distance (for example, the euclidean norm) in both cases. Of course,
this ``decoupling" between cubes and the $\ell_2$ norm means
that we are outside the scope of Theorem \ref{mainabst}. However,
the maximal operator associated to cubes fits well with the
product structure of $\mathbb{R}^d$, so it is natural, and
common, to
use it together with the natural distance on $\mathbb{R}^d$, the euclidean length.
The next question seems therefore interesting to us:

\vskip .2cm

{\bf Question 4.} If in Theorem \ref{cubes} we keep the maximal
operator associated to cubes but   consider
the $\ell_2$ instead of the $\ell_\infty$ norm, how do
the conclusions change?

\section
{Local  results in one dimension.}

Next we study the local case in one dimension, that is, when
the domain is a proper subinterval of $\mathbb{R}$. While the
idea of the proof is the same as in the previous results, formally
the next theorem does not follow from them, so we include the full argument.
In fact, notation is considerably simplified by the fact that only intervals with
$x$ as one endpoint need to be considered when computing $Mf(x)$.
In order to define any such  interval it is enough to specify the other extreme, there is no need
to talk about centers, radii and the relations between them.

We note that constants are worse in the local case  than in the
global case due, to the
fact that a proper subinterval $I$ of $\mathbb{R}$ has at least one
boundary point, say, for instance a left endpoint $a$. Then the
decay restrictions imposed by the modulus of continuity only hold
to the right of $a$, and so the level sets of an extremal function will
be in general smaller than they would be if the function were defined
over the whole real line. This is reflected in the different integral
formulas; recall that for $\mathbb{R}$, Theorem \ref{cubes} tells
us that
\begin{equation}\label{modintR}
\omega\left(Mf,t\right)\le \min_{0\le s\le 1}
\frac{1}{1 + s} {\displaystyle\int_{-s}^{1}}
\omega\left(|f|,  t u\right)du.
\end{equation}
The existence of a boundary point entails that we must take $s=0$ in the
local case, cf. (\ref{modint}) and (\ref{modintsh}) below.

\begin{remark}\label{bddint} If $f$ is uniformly continuous on a {\em bounded}
interval $I$, then its modulus of continuity is constant on
$[|I|, \infty)$, with value $\omega(f,|I|)$. Thus,
if we are given a modulus $\omega$, and $|I| < \infty$, we cannot
 expect to find a function $\psi\in  \operatorname{Lip} (\omega , I)$ with $\omega (\psi, t) =\omega(t)$ for all $t > 0$.  In general, the best we can do is to find $\psi$ so that $\omega (\psi, t) =\omega(t)$
on $(0, |I|]$.
\end{remark}

\begin{remark} By an interval $I$ we always mean a nondegenerate
interval, so the empty set $(a,a]$ and points $[a,a]$ are excluded. Also, the
requirement that subintervals be proper leaves out
the already studied case $I =\mathbb{R}$.
Other than that, there are no restrictions on the subintervals
of $\mathbb{R}$ considered in Theorem \ref{main1d}. Nevertheless,
for convenience we shall assume {\em in the proof} that intervals are
not open. The open case is handled in essentially the same way,
via a limit argument. So $I$ will contain at least one endpoint.
To simplify notation, we shall assume that this endpoint is
the origin, and furthermore, that it is the left endpoint of $I$.
\end{remark}

\begin{theorem}\label{main1d}  Let $I\subset\mathbb{R}$ be a proper subinterval,
and let $f:I\to \Bbb R$ be locally integrable. Then, for every
$t>0$,
\begin{equation}\label{modint}
\omega\left(Mf,t\right)\leq {\displaystyle\int_{0}^{1}}
\omega\left(|f|,  t u\right)du.
\end{equation}
Inequality (\ref{modint}) is sharp in the sense that for every
modulus of continuity $\omega$ and every $\delta > 0$ with $\delta \le |I|$, there exists
a nonnegative, uniformly continuous function $\psi:I\to\mathbb{R}$
such that for all $t\in (0,\delta)$, we have $\omega (\psi, t) =\omega(t)$
and
\begin{equation}\label{modintsh} \omega\left(M\psi,t\right) =
{\displaystyle\int_{0}^{1}} \omega\left( t u\right)du.
\end{equation}
Moreover, if $\omega$ is bounded, then the   function
$\psi$ can be chosen so that it satisfies $\omega (\psi, t) =\omega(t)$ and (\ref{modintsh})  for all $t > 0$ with $t < |I|$.
\end{theorem}

\begin{proof}
 Assume  $f\ge 0$, and note that when we evaluate $Mf(x)$, taking the supremum over all
intervals containing $x$ yields the same value as taking the
supremum over intervals having $x$ as a boundary point (in other
words, $Mf$ is the maximum of the right and left one sided maximal
functions). Let $x,y\in I$ be such that $Mf(y)<Mf(x)$. Then
 \begin{equation*}
Mf(x)-Mf(y) =\sup_{S\in I}\dfrac{1}{S-x}{\displaystyle\int_{x}^{S}}
f(u)du-Mf(y)
 \end{equation*}
 \begin{equation*}
\leq \sup_{S\in I}\left(\dfrac{1}{S-x}{\displaystyle\int_{x}^{S}}
f(u)du-\dfrac{1}{S-y}{\displaystyle\int_{y}^{S}} f(u)du\right)
\end{equation*}
\begin{equation}\label{leo}
\le \sup_{S\in I}{\displaystyle\int_{0}^{1}}
\left|f(x+\left(S-x\right)u)-f(y+\left(S-y\right)u)\right|du \le
\int_{0}^{1} \omega(\left|f\right|, |x-y|u)du.
\end{equation}
Now using a symmetry argument between $x$ and $y$, and taking the
supremum over $|x - y| \le t$, (\ref{modint}) follows.

We prove the optimality of (\ref{modint}) on the interval
$I=[0,\infty)$. The argument can be easily adapted to other proper
subintervals of $\mathbb{R}$.  If $\omega$ is bounded, set $\psi:=
\|\omega\|_{L^\infty([0,\infty ))} - \omega$. It is immediate from
the definitions that $\omega (\psi,t ) = \omega (t )$. Since $\psi$
is decreasing, $M\psi(t) = t^{-1}\int_0^t \psi(u)du = \int_0^1
\psi(tu)du$, so
\begin{equation}\label{bddt}
\omega\left(M\psi,t\right)\ge M\psi(0) - M\psi(t) =
\|\omega\|_{L^\infty([0,\infty ))} - {\displaystyle\int_{0}^{1}}
\psi(t u )du = {\displaystyle\int_{0}^{1}} \omega\left( t
u\right)du.
\end{equation}
Thus, by (\ref{modint}) we have equality. If $\omega$ is unbounded,
the result follows from the previous case by fixing $\delta > 0$ and
considering the bounded modulus $\omega_\delta (t) :=
\min\{\omega(t),\omega(\delta)\}$.
\end{proof}

\begin{remark} In (\ref{bddt}) we are crucially using  that $t <|I|$. Suppose $I = [0,1]$, and let $\omega (t) = \omega (\psi, t) = \min\{t,1\}$, where $\psi (x) = 1 - x$. Then for all $t\ge 1$,
$\omega (M\psi, t) = 1/2 < 1 = \lim_{t\to\infty} \int_{0}^{1} \omega\left( t
u\right)du.$ For the same reason, the restriction $t <|I|$ is also
needed in (\ref{localop}) below.
\end{remark}

Recall  that $
\|f\|_{\operatorname{Lip}(\omega, I)} =
\sup_{t > 0}\omega(f,t)/\omega(t)
$. In the next result we use the following notation:
$\psi _T: [0,\infty ) \to [0,\infty)$ is defined by
$\psi_T:=
(T - \omega)^+$. We shall see that only the functions
$\psi_T$ need to be taken into account when computing
$\|M\|_{\operatorname{Op} (\omega )}$. If $\omega$ is bounded
we set $T= \|\omega\|_{L^\infty([0,\infty ))}$ (as usual) and
just write $\psi$.

\begin{corollary}\label{lipholloc}  Let $\omega $ be a modulus of
continuity and let $I$ be a proper subinterval of $\mathbb{R}$. On $\operatorname{Lip} (\omega ,I)$,
\begin{equation}\label{localop}
\|M\|_{\operatorname{Op} (\omega )} =
\sup_{0< t < |I|} \frac1{\omega (t)} {\displaystyle\int_{0}^{1}} \omega\left( t u\right)du.
\end{equation}
 In the special case where $\omega (t) = t^\alpha$, $\alpha\in (0,1]$, that
is, over the H\"older and Lipschitz classes,
\begin{equation}\label{lipholop1d}
\|M\|_{\operatorname{Op} (\alpha )} = \frac1{1 + \alpha}.
\end{equation}
\end{corollary}

\begin{proof}  If $\omega$ is bounded, (\ref{localop}) follows
immediately
from the fact that there is an extremal function $\psi$ such that   for all $t > 0$ with $t < |I|$,
 $\omega (\psi, t) =\omega(t)$ and
$ \omega\left(M\psi,t\right) =
\int_{0}^{1} \omega\left( t u\right)du.
$ And if $|I| < \infty$, replacing $\omega$ by
$\omega^\prime :=  \min\{ \omega, \omega (|I|)\}$,
we are back to the bounded case.

Suppose next that both $\omega$ and $I$ are unbounded. Without loss
of generality we may assume, in order to simplify notation, that $I = [0,\infty)$.  Select for each $n\ge 1$ a norm one
function $f_n \in \operatorname{Lip} (\omega ,I)$  approximating
the operator norm of $M$ to within $1/n$. Then
$$
\|M\|_{\operatorname{Op} (\omega )} \le \|M f_n\|_{\operatorname{Lip} (\omega )} + \frac{1}{n}  \le \sup_{t > 0} \frac1{\omega (t)} {\displaystyle\int_{0}^{1}} \omega\left(|f_n|, t u\right)du + \frac{1}{n}
\le \sup_{t > 0 } \frac1{\omega (t)} {\displaystyle\int_{0}^{1}} \omega\left( t u\right)du + \frac{1}{n}.
$$
On the other hand, since each $\psi_T$ has norm 1,
$$
\|M\|_{\operatorname{Op} (\omega )} \ge  \sup_{T > 0} \|M \psi_T\|_{\operatorname{Lip} (\omega )} =
\sup_{T > 0}\left\{ \sup_{0< t < T} \frac1{\omega (t)} {\displaystyle\int_{0}^{1}} \omega\left(\psi_T,  t u\right)du\right\}
$$
$$
 =
\sup_{T > 0} \left\{ \sup_{0< t < T} \frac1{\omega (t)} {\displaystyle\int_{0}^{1}} \omega\left( t u\right)du \right\}
= \sup_{t > 0} \frac1{\omega (t)}
{\displaystyle\int_{0}^{1}} \omega\left( t u\right)du .
$$

If $\omega (t) = t^\alpha$, $\alpha\in (0,1]$, then
\begin{equation*}
\|M\|_{\operatorname{Op} (\omega )} =
\sup_{0< t < |I|} \frac1{\omega (t)} {\displaystyle\int_{0}^{1}} \omega\left( t u\right)du = \int_0^1  u^\alpha du = \frac1{1 + \alpha}.
\end{equation*}
\end{proof}

\begin{remark}\label{landa} The sharp bounds given above yield information
about related inequalities. For instance, in Theorem 5.1 of
\cite{AlPe} the following Landau type inequality is proven: If
$u:[0,\infty) \to \mathbb{R}$ is
 an absolutely continuous function such that
its derivative is of bounded variation, then \begin{equation}
    \label{landau}
    \|u'\|_\infty^2 \le 48 \|u\|_\infty \left( \|DM(u'^+)\|_\infty +
    \|DM(u'^-)\|_\infty\right).
\end{equation}
 It is natural to seek a lower bound on the best possible constant $c$ that can
 replace $48$
in (\ref{landau}).   Recall that the classical (sharp) Landau
inequality assumes more regularity on the
 part of  $u^\prime$:
If it is absolutely continuous on $[0,\infty)$, then
 \begin{equation}\label{reallandau}
  \|u'\|_\infty^2\le 4\|u\|_\infty\|u''\|_\infty.
\end{equation} As noted in Remark 5.5, of \cite{AlPe}, (\ref{landau})
implies Landau's inequality, save for the issue of best constants.
Using Corollary \ref{lipholloc} and (\ref{landau}) we have that
$\|DM(u'^+)\|_\infty  + \|DM(u'^-)\|_\infty \le 2^{-1}\|D
u'^+\|_\infty  + 2^{-1}\|D u'^-\|_\infty \le \|D u'\|_\infty$, so
$4\le c\le 48$. On the whole real line, the constant from Theorem
5.1 of \cite{AlPe} appearing in (\ref{landau}) is 24 instead of 48,
and the best constant in the sharp Landau's inequality
(\ref{reallandau}) is 2 instead of 4. Arguing as before and using
Corollary \ref{lip1}, we have $\|DM(u'^+)\|_\infty  + \|DM(u'^-)\|_\infty \le
2 (\sqrt2 - 1)\|D u'\|_\infty$, so $1 /(\sqrt2 - 1)\le c\le 24$.
Thus, on $\mathbb{R}$ the best constant appearing in the generalized
Landau inequality (\ref{landau}) is strictly larger than the best
constant 2 in the classical Landau inequality.
\end{remark}

\begin{remark}\label{ex} By   Corollary \ref{lipholloc} the functions $\psi_\alpha (x):=\max\{1-x^\alpha,0\}$ are extremal
 in the classes $\operatorname{Lip}(\alpha) ([0,\infty))$.
  Now $M \psi_\alpha (x)=
1-\dfrac{x^{\alpha}}{1+\alpha}$ on $0\le x<1$  (and $M \psi_\alpha (x)=
\dfrac{\alpha }{(1+\alpha)x}$ on $1\le x <\infty$). Thus,
the H\"older or Lipschitz exponent of $M\psi_\alpha$ is no better than that
of $\psi_\alpha$.  But in this
example there is only one ``bad point": For every $\varepsilon > 0$,
both $\psi_\alpha$
and $M\psi_\alpha$ are Lipschitz on $[\varepsilon, \infty)$. In view of
the regularizing properties of $M$, one may wonder, for instance, whether $Mf$ must have
a better H\"older exponent than $f$ save for small sets, or
 in some ``almost everywhere" sense. We shall see below that the answer
is negative.

We also note that the preservation of regularity  does not extend to the $C^r$ classes. To
see this, let $f$ be the sum of two smooth bump functions with
disjoint supports and recall that $Mf$ is the maximum of the right
and left one sided maximal functions. This entails that between the
bumps of $f$,  $Mf$ achieves its minimum value at a point of non-differentiability, since  the right and left derivatives of $Mf$
are nonzero there and have opposite signs.
\end{remark}

Next we show that  a H\"older condition on
$f$ is not sufficient to ensure the differentiability almost
everywhere of $Mf$, so in particular
 $Mf$ can fail to be absolutely continuous. The following
  example also shows that the H\"older exponent of $Mf$
may be as bad as that of $f$ on a large set.
To prove this
we suitably modify the fat Cantor set defined in  Example 4.2 of \cite{AlPe}, and the functions used there.

\begin{example}\label{cantor} {\em
For every $\alpha \in (0,1)$ there exists a  function
 $f\in \operatorname{Lip}_\alpha ([0,1])$, with
$\operatorname{Lip}_\alpha (f) = 1$,  such that
$Mf$ is not differentiable on a set $E$ of positive measure.
Furthermore, given any $\beta \in (\alpha , 1)$, $Mf$ is not
locally H\"older $(\beta)$ at any point of $E$. More precisely,
given
any $x\in E$, and any   interval $I$, relatively open in
$[0,1]$ and with $x\in I$, we have
$Mf\notin \operatorname{Lip}_\beta (I)$}.

\vskip .2 cm

{\em Proof:}   As  in Example 4.2 of \cite{AlPe},
and unlike the
usual construction of the Cantor set,  instead of removing the ``central part" of every
interval at each stage, we
remove several parts. Let $F_0 = [0,1]$ and let
$F_n$ be the finite union of closed subintervals of $[0,1]$
obtained at step $n$ of the construction, to be described below. As usual
 $C:= \cap _n F_n$. Denote by $\ell (J)$ and $r(J)$
the left and right endpoints of an interval $J$. At stage $n$ we remove
 $2^{-2n}$ of the mass in the preceding set, so  $|C|  >
 1 - |\cup_1^\infty F_n^c| \ge 1 - \sum_1^\infty 2^{-2n} =  2/3$.
  Let $I_{n-1}$ be a
component of $F_{n-1}$, and let
$O(n,x)$ the open interval centered at $x$ of length
$2^{-4n}|I_{n-1}|$. Select the unique $2^{2n}$
 points $x_1,x_2 \dots ,x_{2^{2n}}\in I_{n-1}$ such that
i) $\ell(I_{n-1}) < x_1 <\dots  < x_{2^{2n}} < r(I_{n-1})$,
and ii) after the removal of the $2^{2n}$ open intervals
$O(n, x_{i})$ from $I_{n-1}$, the $2^{2n} + 1$ remaining
components have equal length, which therefore must be $(1-2^{-2n}) |
I_{n-1}|/(2^{2n}+1)$. Thus, $x_2 - x_1 = x_{i+1} - x_i$ for $i <
2^{2n}$, while $r(I_{n-1}) - x_{2^{2n}} = x_1 - \ell (I_{n-1})
< x_2 - x_1$. Note that the extremes of
$I_{n-1}$ are left untouched.

Repeat the same process with  the
other components of $F_{n-1}$ to obtain $F_n$, which by construction
is a disjoint union of closed intervals, all of length
 $(1-2^{-2n}) |I_{n-1}|/(2^{2n}+1)$.

Next we define the function $f$.
 Given $t > 0$,
set $g_t(x) = 1 - (t-x)^\alpha$ on $[0,t)$ and  extend it to
$(-t,t)$ as an even function. Then let $f$  be identically $1$ on
$C$, and let $f(x):= g_t(x - s)$ on each
 interval of type $O(j,s)$, where $t$ is half the length of $O(j,s)$.
To see why $f\in \operatorname{Lip}_\alpha ([0,1])$ and
$\operatorname{Lip}_\alpha (f) = 1$, note that if $x\in C$ and
$y\in [0,1]\setminus C$, say, with $y < x$, then there is a unique
$O(n,w)$ containing $y$, so we can replace $x$ by $r(O(n,w))$ and
then $f(x) - f(y) = f(r(O(n,w))) - f(y)$, while $r(O(n,w)) - y \le x
- y$; the case where  $x$ and $y$ belong to different components of $[0,1]\setminus C$  can be reduced
to the previous one: Let $y \in O(n,u)$,  $x \in O(m,w)$, and suppose $y < x$.   If $f(x) < f(y)$, replace $y$
with $\ell (O(m,w))$; otherwise, replace $x$ with $r(O(n,u))$. So it
is enough to verify that $f$ has the desired properties on the
closure of an arbitrary $O(j,s)$, and this is true since it holds
for $|x|^{\alpha}$, its translates, reflections, etc. Next, let $f_n:= 1$ on $F_n$ and $f_n := f$ on
$[0,1]\setminus F_n$. Then $f\le f_n$, so $Mf\le Mf_n$, and of
course the latter function is easier to work with. Let
$x_i$ and $O(n,x_i)$ be as above. For notational convenience, suppose
we are taking $i=2$; the same reasoning applies to the other $x_j$.
To evaluate $Mf_n (x_2)$ we only need to consider intervals
$J$ that have $x_2$ as an endpoint and, we claim,
 omit both $x_1$ and $x_3$. To see this,  assume that
 $x_2 = \ell (J)$. Since $f_n (x_2) < Mf_n (x_2) < 1$, an optimal
 $J$ must satisfy $f_n(r(J)) = Mf_n (x_2)$. Otherwise we could improve
the average by either shortening or enlarging $J$. Thus,
$r(J)\in O(n,x_3)\cap\{x < x_3\}$. Of course, the situation is
completely symmetrical if $x_2 = r(J)$ (here
$\ell(J)\in O(n,x_1)\cap\{x > x_1\}$), so we keep the
assumption $x_2 = \ell (J)$. Next, define $h\ge f_n$ by setting
$h\equiv 1$ on $O(n,x_3)\cap\{x < x_3\}$ and $h\equiv f_n$ elsewhere.
Clearly
$$
Mf_n (x_2) < \frac{1}{x_3 - x_2}\int_{x_2}^{x_3} h = \frac{(\frac{1-2^{-2n}}{2^{2n}+1}+2^{-4n-1})
|I_{n-1}| + \int_0^{2^{-4n-1} | I_{n-1}|} 1 - x^\alpha
dx}{(\frac{1-2^{-2n}}{2^{2n}+1}+2^{-4n}) |I_{n-1}|}
$$
$$
= 1-\frac{2^{(-4n-1)(\alpha+1)}
}{(\frac{1-2^{-2n}}{2^{2n}+1}+2^{-4n})(\alpha+1)}|I_{n-1}|^\alpha
<  1-2^{-4\alpha n-2n - \alpha -2}|I_{n-1}|^\alpha.
$$
Fix $z\in C\cap \{Mf =1\}$. For each $n$, let $I_{n,z}$ be the
component of $F_n$ that contains $z$, and let $w_n$ be midpoint of
the nearest  interval $O(n)$ deleted at step $n$ (if there are two
such midpoints, pick any). Then $|z - w_n|< 2^{-2n} | I_{n-1}|$, and
furthermore $|I_{n-1}| < 2^{-n(n-1)}$, since the number of
components of the
 set $F_{n-1}$ is  $\prod_{i=1}^{n-1} (2^{2i}+1) >
 \prod_{i=1}^{n-1} 2^{2i} = 2^{n(n-1)}$. Thus, for every $\beta\in (\alpha,1]$ we have
$$
\limsup_{w\to z} \frac{\left|Mf(z) -
Mf(w)\right|}{\left|z-w\right|^\beta} \ge \limsup_{n\to \infty}
\frac{1 - Mf_n(w_n)}{|z-w_n|^\beta} \ge \lim_{n\to \infty}
\frac{2^{-4\alpha n-2n - \alpha -2}}{2^{-2\beta n}(2^{-
n(n-1)})^{\beta-\alpha}}= \infty.
$$
Since $Mf \ge f$ a.e., the sets $E:= C\cap \{Mf =1\}$ and $C$ have the same measure, so
taking $\beta = 1$ it follows that $Mf$ is not
differentiable on a set of measure at least $2/3$,
while for $\beta\in (\alpha,1)$ we obtain the failure of the
local H\"older condition with exponent $\beta$.
\qed
\end{example}

\begin{remark}\label{localhidee}   The following example
shows that  local, higher dimensional generalizations of the
preceding results may fail. Recall that in Definition \ref{maxfun}
the balls in $\mathbb{R}^d$ we average over, are assumed to be fully
contained in the (open) domain $U$ of $f$. Let
$U=\left\{\left(x,y\right):\;x\in\mathbb{R}
^{d},\;y\in\mathbb{R},\;y>\exp\left(-1/\left|x\right|\right)
\right\}$ be a domain with a cusp and let $f(x,y)=y$. Then
$\lim_{y\downarrow 0} Mf(0,y)=0$, while $Mf(0,1) =\infty$.
\end{remark}

\end{document}